\documentclass[final]{siamart190516}
\usepackage{amsfonts}
\usepackage{amsfonts,amsmath,amssymb}
\usepackage{mathrsfs,mathtools,stmaryrd,wasysym}
\usepackage{enumerate}
\usepackage[mathscr]{euscript} 
\usepackage{enumitem}
\usepackage{comment}  
\usepackage{esint}
\usepackage{graphicx,psfrag}
\usepackage{hyperref}
\usepackage{amsmath,stackengine}
\usepackage{yfonts}
\usepackage{enumitem}
\DeclareMathAlphabet{\mathpzc}{OT1}{pzc}{m}{it}

\newsiamremark{remark}{Remark}

\newcommand{\TheTitle}{Finite element error estimates for a bilinear optimal control problem with pointwise tracking governed by a semilinear elliptic PDE}
\newcommand{\ShortTitle}{Bilinear optimal control with pointwise tracking}
\newcommand{\TheAuthors}{E. Ot\'arola, D. Quero, M. Sasso}

\headers{\ShortTitle}{\TheAuthors}

\title{{\TheTitle}\thanks{EO is partially supported by USM through USM project 2026 PI\_LIR\_26\_06. MS is supported by ANID through Subdirección del Capital Humano/Doctorado Nacional/2026-21262084 and by Direcci\'on de Postgrado of UTFSM through Programa de Incentivos a la Investigaci\'on Cient\'ifica (PIIC) No.~057/2025.}}

\author{Enrique Ot\'arola\thanks{Departamento de Matem\'atica, Universidad T\'ecnica Federico Santa Mar\'ia, Valpara\'iso, Chile. \email{enrique.otarola@usm.cl}}
\and
 Daniel Quero\thanks{Departamento de Ciencias Exactas, Universidad de Los Lagos, Osorno, Chile \email{daniel.querotangol@ulagos.cl}}
\and
Mat\'ias Sasso\thanks{Departamento de Matem\'atica, Universidad T\'ecnica Federico Santa Mar\'ia, Valpara\'iso, Chile \email{matias.sasso@sansano.usm.cl}}}

\ifpdf
\hypersetup{
  pdftitle={\TheTitle},
  pdfauthor={\TheAuthors}
}
\fi

\date{Draft version of \today.}

\makeatletter
\AtBeginDocument{
  \def\url#1{\unskip\kern-0.28em.\@ifnextchar.{\@gobble}{}}
  \def\doi#1{\unskip\kern-0.28em.\@ifnextchar.{\@gobble}{}}
}
\makeatother

\begin{document}

\maketitle
\begin{abstract}
We study finite element approximations for an optimal control problem with pointwise tracking governed by a semilinear elliptic PDE. The control variable enters the state equation as a reaction coefficient, resulting in a bilinear control, and the cost functional includes point evaluations of the state variable. We propose two discretization schemes: a semidiscrete scheme, where the control variable is not discretized, and a fully discrete scheme, where the control variable is discretized using piecewise constant functions. In two and three dimensions, for both schemes, we establish convergence of discrete solutions and prove a priori error bounds that behave as $\mathcal{O}(h|\log h|)$ in the $L^2(\Omega)$-norm for approximating a locally optimal control. In two dimensions, these bounds are improved to $\mathcal{O}(h)$ for the fully discrete scheme and $\mathcal{O}(h^{2}|\log h|^2)$ for the semidiscrete scheme. We conclude with numerical experiments that demonstrate the performance of the proposed schemes.
\end{abstract}



\section{Introduction}
\label{sec:intro}
In this paper, we design and analyze finite element methods (FEMs) for a bilinear optimal control problem governed by a semilinear elliptic partial differential equation (PDE). The control variable enters the state equation as a reaction coefficient and is required to satisfy box constraints, while the cost functional tracks the state variable at a finite collection of points. To make the setting precise, we let $d \in \{2,3\}$, let $\Omega \subset \mathbb{R}^{d}$ be an open, bounded polytopal domain with Lipschitz boundary $\partial \Omega$, and let $\mathcal{D} \subset \Omega$ be a finite ordered set. The convexity of $\Omega$ will be assumed from Section \ref{sec:finite_element_disc} onward. Given a regularization parameter $\alpha >0$ and a set of desired states $\{y_{t}\}_{t \in \mathcal{D}} \subset \mathbb{R}$, we introduce the cost functional
\begin{equation}\label{def:cost_func}
J(y,u) := \dfrac{1}{2}\sum_{t \in \mathcal{D}}(y(t) - y_{t})^{2} + \dfrac{\alpha}{2}\|u\|^{2}_{L^{2}(\Omega)}.
\end{equation}
Let $f \in L^2(\Omega)$ be given. The optimal control problem under consideration reads: Minimize $J(y,u)$ subject to the \emph{semilinear elliptic} PDE
\begin{equation}\label{def:state_eq}
-\Delta y + a(\cdot,y) + uy = f
\text{ in } \Omega,
\qquad y = 0
\text{ on } \partial \Omega,
\end{equation}
and the \emph{control constraints}
\begin{equation}\label{def:box_const}
u \in \mathbb{U}_{ad},
\qquad
\mathbb{U}_{ad}:= \{v \in L^{2}(\Omega): \mathtt{a} \leq v(x) \leq \mathtt{b} \text{ for a.e.}~x \in \Omega\}.
\end{equation}
Here, $-\infty < \mathtt{a} < \mathtt{b} < \infty$ are the control bounds, and $a: \Omega \times \mathbb{R} \rightarrow \mathbb{R}$ is a Carath\'eodory function with precise assumptions stated in Section~\ref{sec:assumptions}.

In problem \eqref{def:cost_func}--\eqref{def:box_const}, the control variable $u$ acts as a coefficient in the state equation rather than as an external forcing term. This gives rise to the nonlinear coupling $uy$ between the control and the state, which characterizes \eqref{def:cost_func}--\eqref{def:box_const} as a \emph{bilinear optimal control problem} \cite{MR0414174,MR2641453,MR3878305,Casas_bilinear_1,MR4956345}.
Bilinear controls are particularly relevant when the objective is to modify intrinsic properties of the system \cite{MR735811,DUCA2021109324}; examples include neutron transport \cite{MR735811}, cancer chemotherapy \cite{ledzewicz2004,Zerrik_etal}, and bilinear quantum systems \cite{DUCA2021109324}. Mathematically, even in the linear case $a \equiv 0$, the solution of \eqref{def:state_eq} depends nonlinearly on $u$. Consequently, the resulting optimization problem is nonconvex, global uniqueness of optimal controls cannot be expected, and second order optimality conditions become essential for characterizing local minima and deriving error estimates for FEMs \cite{MR2536007,Casas_bilinear_1,MR4956345}. Moreover, the bilinear term $uy$ may destroy the coercivity of the forms associated with the state and adjoint equations. This requires adapting standard arguments to establish well-posedness and differentiability properties of the control-to-state and control-to-adjoint state maps \cite{Casas_bilinear_1,MR5008466} and introduces additional difficulties in the analysis of the corresponding FEMs \cite{MR2536007,MR4956345}.

  A second distinctive feature of \eqref{def:cost_func}--\eqref{def:box_const} is the pointwise tracking structure of $J$, which involves point evaluations of the state $y$ at the locations $t \in \mathcal{D}$. This produces an adjoint problem whose forcing term is a linear combination of Dirac measures and, consequently, an adjoint variable $p$ with reduced Sobolev regularity: $p \in W^{1,r}(\Omega)$ for $r < d/(d-1)$; see \cite{MR3449612,MR3800041,MR3973329,MR4438718}. At the continuous level, this reduced regularity complicates the analysis, since $p$ enters the optimality conditions through the product $yp$ and its singular behavior must be taken into account when establishing regularity properties of locally optimal controls. At the discrete level, the reduced regularity directly affects the error analysis: suitable $L^2(\Omega)$ error estimates are needed for the adjoint problem, while the comparison of the continuous and discrete adjoint equations involves pointwise errors of the state, making $L^\infty(\Omega)$ error estimates for the state essential.
  The theoretical foundations for \eqref{def:cost_func}--\eqref{def:box_const}, including well-posedness, optimality conditions, and regularity of locally optimal controls, have recently been established in \cite{MR5008466}, on which the present work builds.

Regarding FEMs for bilinear optimal control problems, a pioneering contribution is the work by Kr\"oner and Vexler \cite{MR2536007}, where the case $a \equiv 0$ is analyzed and convergence rates $\mathcal{O}(h)$ and $\mathcal{O}(h^{3/2})$ for the approximation of an optimal control are obtained using piecewise constant and piecewise linear discretizations, respectively. More recently, the authors of \cite{MR4956345} have developed a FEM, closely related to the fully discrete scheme proposed in this work, for a bilinear optimal control problem governed by a semilinear PDE with a smooth, distributed cost functional. To the best of our knowledge, FEMs for the bilinear optimal control problem with pointwise tracking \eqref{def:cost_func}--\eqref{def:box_const},  which, as discussed above, requires substantially different regularity and error analyses, have not previously been considered. In the related but distinct setting of additive controls with pointwise tracking, we refer to \cite{MR3449612,MR3523574,MR3800041,MR3973329,MR4438718} for a priori error estimates and to \cite{MR3679932} for adaptive methods.

In this work, we develop and analyze two finite element discretization strategies for problem \eqref{def:cost_func}--\eqref{def:box_const}: a fully discrete scheme, in which the control variable is approximated by piecewise constant functions, and a semidiscrete scheme based on the variational discretization concept of \cite{MR2122182}, in which the control space is not explicitly discretized. For both schemes, we prove convergence of discrete solutions and derive a priori error estimates of order $\mathcal{O}(h|\log h|)$ in the $L^{2}(\Omega)$-norm for approximating a locally optimal control. In two dimensions, these estimates are improved to $\mathcal{O}(h)$ for the fully discrete scheme and $\mathcal{O}(h^2|\log h|^{2})$ for the semidiscrete scheme. The analysis is complemented by numerical experiments in two and three dimensions that are consistent with the theoretical rates.

In what follows, we briefly discuss the main contributions of this work and the main difficulties involved in establishing them:
\begin{itemize}[leftmargin=*,nosep]
\item \emph{Regularity of locally optimal controls $(d=3)$:} If $\Omega$ is a convex polytope, we extend the analysis of \cite{MR5008466} to three dimensions and establish, under suitable sign assumptions on $f-a(\cdot,0)$, the Lipschitz regularity of a locally optimal control $\bar{u}$.
\item \emph{Convergence of discretizations:} We prove that every sequence of discrete global solutions admits a subsequence converging, in the $L^{2}(\Omega)$-norm, to a global solution of \eqref{def:cost_func}--\eqref{def:box_const}, and that continuous strict local solutions can be approximated by local solutions of the discrete problems.
\item \emph{Error estimates:} For both discretization schemes, we derive error bounds that behave as $\mathcal{O}(h|\log h|)$ in the $L^{2}(\Omega)$-norm for approximating a locally optimal control. In two dimensions, we improve these estimates to $\mathcal{O}(h)$ for the fully discrete scheme and $\mathcal{O}(h^{2}|\log h|^{2})$ for the semidiscrete scheme. The proof of the latter relies on several refined results concerning the pointwise behavior of the discrete adjoint state near the observation points at which the optimal state does not attain the desired value $y_{t}$, and the coincidence of the continuous and semidiscrete optimal controls in neighborhoods of these points.
\end{itemize}

The paper is organized as follows. In Section \ref{sec:notation_and_prel}, we introduce the notation and assumptions used throughout the paper. Preliminary results concerning weak solutions of \eqref{def:state_eq} are reviewed in Section \ref{sec:semilinear_eq}. In Section \ref{sec:OCP}, we review the main results of \cite{MR5008466}: existence of solutions and optimality conditions for a weak version of \eqref{def:cost_func}--\eqref{def:box_const}. In Section \ref{sec:loc_opt_control_reg}, we collect and extend regularity properties of locally optimal controls. In Section \ref{sec:finite_element_disc}, we introduce two numerical schemes for \eqref{def:cost_func}--\eqref{def:box_const} and analyze the discretization of the corresponding state and adjoint equations. Convergence of the resulting discrete optimal control problems is established in Section \ref{sec:conv_discretizations}, while Section \ref{sec:error_estimates} is devoted to a priori error estimates for the approximation of a locally optimal control. Finally, in Section \ref{sec:num_examples}, we present numerical experiments in two and three dimensions that illustrate the theoretical results.


\section{Notation and assumptions}
\label{sec:notation_and_prel}

This section introduces the notation and assumptions on which the subsequent analysis is based.


\subsection{Notation}\label{sec:notation}
Throughout this work, let $d \in \{2,3\}$ and let $\Omega \subset \mathbb{R}^{d}$ be an open, bounded, polytopal domain with Lipschitz boundary $\partial \Omega$. The convexity of $\Omega$ will be assumed when needed; in particular, for certain regularity results and throughout the finite element analysis.
 
Let $\mathscr{X}$ and $\mathscr{Y}$ be Banach function spaces. The dual space of $\mathscr{X}$ is denoted by $\mathscr{X}'$ and the norm of $\mathscr{X}$ by $\|\cdot\|_{\mathscr{X}}$. The duality pairing between $\mathscr{X}'$ and $\mathscr{X}$ is denoted by $\langle \cdot,\cdot\rangle_{\mathscr{X}',\mathscr{X}}$ and abbreviated as $\langle \cdot,\cdot\rangle$ when no confusion arises. The continuous embedding of $\mathscr{X}$ into $\mathscr{Y}$ is denoted by $\mathscr{X} \hookrightarrow \mathscr{Y}$. For $\{x_{n}\}_{n \in \mathbb{N}} \subset \mathscr{X}$, we use $x_{n} \rightarrow x$ and $x_{n} \rightharpoonup x$ to denote strong and weak convergence to $x$, respectively, as $n \uparrow \infty$.

The notation $\mathfrak{a}\lesssim\mathfrak{b}$ indicates that $\mathfrak{a}\leq C\mathfrak{b}$, where $C>0$ is a generic constant independent of the quantities under consideration. The value of $C$ may change at each occurrence and will be specified explicitly whenever relevant.


\subsection{Assumptions}\label{sec:assumptions}
The analysis presented in this paper relies on the following assumptions
on the nonlinear function $a$, which have also been considered in \cite{MR5008466,Casas_bilinear_1}:
\begin{enumerate}[label=(A.\arabic*)]
\item \label{A1} $a : \Omega \times \mathbb{R} \rightarrow \mathbb{R}$ is a Carath\'eodory function of class $C^{2}$ with respect to its second variable and $a(\cdot , 0) \in L^{2}(\Omega)$.
\item \label{A2} There exists a function $a_{0} \in L^{\infty}(\Omega)$ such that $\frac{\partial a}{\partial y}(x,y) \geq a_{0}(x)$ for a.e.~$x \in \Omega$ and for all $y \in \mathbb{R}$.
\item \label{A3} For all $\mathfrak{m}>0$, there exists a constant $C_{a,\mathfrak{m}}>0$ such that
\begin{equation*}
\sum_{i=1}^{2} \left| \frac{\partial^{i} a}{\partial y^{i}}(x,y) \right| \leq C_{a,\mathfrak{m}},
\qquad
\left| \frac{\partial^{2} a}{\partial y^{2}}(x,v) - \frac{\partial^{2} a}{\partial y^{2}}(x,w) \right| \leq C_{a,\mathfrak{m}}\left| v-w\right|
\end{equation*}
for a.e.~$x \in \Omega$ and for all $y,v,w \in [-\mathfrak{m},\mathfrak{m}]$.
\end{enumerate}


\section{The state equation}\label{sec:semilinear_eq}
Throughout this section, we use $\mathtt{u}$, $\mathtt{y}$, and $\mathtt{f}$ to denote the reaction coefficient, solution, and forcing term, respectively, of the state equation considered independently of the optimal control problem. The corresponding quantities in the optimal control problem are denoted by $u$, $y$, and $f$.

Following \cite{Casas_bilinear_1,MR5008466}, we introduce the following set:
\begin{equation}
\mathcal{A}_{0} := \{ w \in L^{2}(\Omega): a_{0}(x)+ w(x) \geq 0~\text{for a.e.}~x\in \Omega \},
\label{eq:mathcal_A_0}
\end{equation}
where $a_{0}$ is as in \ref{A2}. Given $\mathtt{u} \in \mathcal{A}_{0}$ and $\mathtt{f} \in L^{q}(\Omega)$ for some $q > d/2$, we consider the following weak formulation of \eqref{def:state_eq}: Find $\mathtt{y} \in H^{1}_{0}(\Omega)$ such that
\begin{equation}\label{eq:weak_semi_bilinear_eq}
\int_{\Omega} \nabla \mathtt{y} \cdot \nabla v \mathrm{d}x +  \int_{\Omega} a(\cdot,\mathtt{y}) v \mathrm{d}x
+  \int_{\Omega} \mathtt{u}\mathtt{y}v \mathrm{d}x = \int_{\Omega} \mathtt{f}v \mathrm{d}x \quad \forall v \in H_{0}^{1}(\Omega).
\end{equation}

The following result establishes the well-posedness of \eqref{eq:weak_semi_bilinear_eq}.

\begin{theorem}[well-posedness]
Let us assume that \ref{A1}--\ref{A3} hold. Given $\mathtt{u} \in \mathcal{A}_0$ and $\mathtt{f} \in L^{q}(\Omega)$ for some $q > d/2$, there exists a unique solution $\mathtt{y} \in H_0^1(\Omega) \cap L^{\infty}(\Omega)$ for \eqref{eq:weak_semi_bilinear_eq}. In addition, $\mathtt{y}$ satisfies the bounds
\begin{equation}
\label{eq:stability_weak_problem}
\| \nabla \mathtt{y}\|_{L^{2}(\Omega)} \leq \| \mathtt{f} - a(\cdot,0) \|_{H^{-1}(\Omega)},
\qquad
\| \mathtt{y} \|_{L^{\infty}(\Omega)} \lesssim \| \mathtt{f} -a(\cdot,0) \|_{L^{q}(\Omega)}.
\end{equation}
The hidden constant in the $L^{\infty}(\Omega)$-estimate is independent of $\mathtt{y}$, $a$, $\mathtt{u}$, and $\mathtt{f}$.
\label{thm:well-posedness-semilinearPDE}
\end{theorem}
\begin{proof}
See \cite[Theorem 3.1]{MR5008466}.
\end{proof}

Under the additional assumption $\mathtt{f} \in L^{2}(\Omega)$, the solution $\mathtt{y}$ enjoys further Sobolev and H\"older regularity. If, in addition, $\Omega$ is convex, then $\mathtt{y}$ belongs to $H^{2}(\Omega)$.

\begin{theorem}[regularity]\label{thm:state_reg}
Let the assumptions of Theorem \ref{thm:well-posedness-semilinearPDE} hold. Assume, in addition, that $\mathtt{f} \in L^{2}(\Omega)$. Then, the unique solution $\mathtt{y} \in H_{0}^{1}(\Omega) \cap L^{\infty}(\Omega)$ of problem \eqref{eq:weak_semi_bilinear_eq} satisfies the following regularity properties:
\begin{itemize}[leftmargin=*,nosep]
\item[(i)] There exists $\kappa >4$ for $d=2$ and $\kappa >3$ for $d=3$ such that $\mathtt{y} \in W^{1,\kappa}(\Omega) \cap C^{0,\varsigma}(\bar{\Omega})$, where $0<\varsigma \leq 1-d/\kappa < 1$, and
$
\|\nabla \mathtt{y}\|_{L^{\kappa}(\Omega)}
+
\|\mathtt{y}\|_{C^{0,\varsigma}(\bar \Omega)}
\lesssim \| \mathtt{f} - a(\cdot,0)\|_{L^{2}(\Omega)} ( 1 + \| \mathtt{u} \|_{L^2(\Omega)} ).
$
\item[(ii)] If, in addition, $\Omega$ is convex, then $\mathtt{y} \in H^{1}_{0}(\Omega) \cap H^{2}(\Omega)$ and
$
 \| \mathtt{y} \|_{H^{2}(\Omega)} \lesssim \| \mathtt{f} -a(\cdot,0)\| _{L^{2}(\Omega)} (1+\|\mathtt{u} \|_{L^{2}(\Omega)}).
$
\end{itemize}
\end{theorem}
\begin{proof}
See \cite[Theorems 3.2 and 3.4]{MR5008466}.
\end{proof}

We conclude this section by introducing the control-to-state map $\mathcal{S}$. There exist an open set $\mathcal{A} \subset L^{2}(\Omega)$ with $\mathcal{A}_{0} \subset \mathcal{A}$ and a $C^2$ map
$
\mathcal{S}: \mathcal{A} \rightarrow H^{1}_{0}(\Omega) \cap C^{0,\varsigma}(\bar{\Omega}),
$
given by $\mathcal{S}(\mathtt{u}) = \mathtt{y}$, where $\mathtt{y}$ solves \eqref{eq:weak_semi_bilinear_eq}; $\varsigma$ is as in Theorem \ref{thm:state_reg}. The construction of $\mathcal{A}$ and the proof of this result can be found in \cite[Theorem 3.5]{MR5008466}. In addition, given $\mathtt{u} \in \mathcal{A}$ and $\mathtt{g} \in L^{2}(\Omega)$, the directional derivative $\mathtt{z} = \mathcal{S}'(\mathtt{u}) \mathtt{g} \in H^{1}_{0}(\Omega) \cap C^{0,\varsigma}(\bar{\Omega})$ solves
\[
\int_{\Omega} \nabla \mathtt{z} \cdot \nabla v \mathrm{d}x + \int_{\Omega} \left[ \frac{\partial a}{\partial y}(\cdot,\mathtt{y}) + \mathtt{u} \right]\mathtt{z} v \mathrm{d}x = -\int_{\Omega} \mathtt{g} \mathtt{y} v \mathrm{d}x \quad \forall v \in H^{1}_{0}(\Omega),
\]
where $\mathtt{y} = \mathcal{S}(\mathtt{u})$; see \cite[Theorem 3.5 and Remark 3.6]{MR5008466} for details.


\section{The optimal control problem}\label{sec:OCP}
Assuming that \ref{A1}--\ref{A3} hold, we consider the following weak version of the optimal control problem \eqref{def:cost_func}--\eqref{def:box_const}:
\begin{equation}\label{eq:weak_ocp}
    \min
    \left\lbrace J(y,u) : (y,u) \in \left( H^{1}_{0}(\Omega) \cap C^{0,\varsigma}(\bar{\Omega}) \right)
    \times \mathbb{U}_{ad}\right\rbrace
\end{equation}
subject to the \emph{semilinear elliptic} PDE
\begin{equation}\label{eq:weak_state_eq}
    (\nabla y, \nabla v)_{L^{2}(\Omega)} + (a(\cdot,y),v)_{L^{2}(\Omega)} + (uy,v)_{L^{2}(\Omega)} = (f,v)_{L^{2}(\Omega)}  \quad \forall v \in H^{1}_{0}(\Omega). 
\end{equation}
Here, $f \in L^{2}(\Omega)$ and $\varsigma$ is as in the statement of Theorem \ref{thm:state_reg}. As in \cite[Assumption 3.1]{Casas_bilinear_1}, we assume that $a_{0}(x)+\mathtt{a} \geq 0$ for a.e.~$x \in \Omega$. Since every $u\in\mathbb{U}_{ad}$ satisfies $u \geq \mathtt{a}$ a.e.~in $\Omega$, we have $\mathbb{U}_{ad}\subset\mathcal{A}_{0}\subset\mathcal{A}$. Consequently, the results of Section \ref{sec:semilinear_eq} apply. In particular, Theorem \ref{thm:state_reg} shows that the unique solution $y$ of problem \eqref{eq:weak_state_eq} belongs to $C^{0,\varsigma}(\bar{\Omega})$, so point evaluations of $y$ are well-defined in the cost functional \eqref{def:cost_func}.
 
We say that $\bar{u} \in \mathbb{U}_{ad}$ is a \emph{globally optimal control} for \eqref{eq:weak_ocp}--\eqref{eq:weak_state_eq} if $J(\mathcal{S}(\bar{u}),\bar{u}) \leq J(\mathcal{S}(u),u)$ for all $u \in \mathbb{U}_{ad}$. In addition, if $\bar{u}$ is a globally optimal control and $\bar{y} = \mathcal{S}(\bar{u})$, then the pair $(\bar{y},\bar{u})$ is called a \emph{globally optimal state-control pair}.

The next result establishes the existence of a globally optimal state--control pair.

\begin{theorem}[global existence]
The optimal control problem \eqref{eq:weak_ocp}--\eqref{eq:weak_state_eq} admits at least one globally optimal state--control pair $(\bar{y},\bar{u}) \in \left( H^{1}_{0}(\Omega) \cap C^{0,\varsigma}(\bar{\Omega}) \right) \times \mathbb{U}_{ad}$.
\end{theorem}
\begin{proof}
See \cite[Theorem 4.1]{MR5008466}.
\end{proof}


\subsection{The adjoint equation}\label{sec:adjoint_eq}
Let $u \in \mathcal{A}$, and let $y = \mathcal{S}(u)$. Let $r$ be such that
\begin{equation}\label{eq:r}
r \in (1,2) ~ \text{if } d=2, \qquad r \in \left[ \tfrac{6}{5},\tfrac{3}{2}\right) ~ \text{if } d=3, 
\end{equation}
and let $s$ be its H\"older conjugate, i.e., $r^{-1} + s^{-1} = 1$. Note that $s>d$. We introduce the \emph{adjoint problem}: Find $p \in W^{1,r}_{0}(\Omega)$ such that
\begin{equation}\label{eq:weak_adj_eq}
\int_{\Omega}\nabla w \cdot \nabla p \mathrm{d}x +  \int_{\Omega}\left[ \frac{\partial a}{\partial y}(\cdot,y) + u \right]pw \mathrm{d}x = \sum_{t \in \mathcal{D}} \langle (y(t)-y_{t})\delta_{t},w \rangle \quad \forall w \in W^{1,s}_{0}(\Omega).
\end{equation} 
Here, $\langle \cdot, \cdot \rangle$ denotes the duality pairing between $W^{-1,r}(\Omega)$ and $W^{1,s}_{0}(\Omega)$. 

\begin{remark}[on the set $\mathcal{A}$]\label{rmk:redef_A}
To establish the well-posedness of problem \eqref{eq:weak_adj_eq}, we redefine the open set $\mathcal{A}$ as in \cite[(4.8)]{MR5008466}. Under this redefinition, the inclusions $\mathbb{U}_{ad} \subset \mathcal{A}_{0} \subset \mathcal{A}$ continue to hold. From now on, $\mathcal{A}$ denotes this redefined set.
\end{remark}

\begin{theorem}[well-posedness]\label{thm:wp_adjoint}
Let $u \in \mathcal{A}$, and let $y = \mathcal{S}(u)$. Then, the weak problem \eqref{eq:weak_adj_eq} has a unique solution $p \in W^{1,r}_{0}(\Omega)$, which satisfies the stability bound
\begin{equation}
\|\nabla p\|_{L^{r}(\Omega)} \lesssim \|f-a(\cdot,0)\|_{L^{2}(\Omega)} + \sum_{t \in \mathcal{D}}|y_{t}|.
\end{equation} 
The hidden constant depends on $\#\mathcal{D}$ and on $u$, but is independent of $p$; in particular, if $u \in \mathbb{U}_{ad}$, it can be chosen uniformly with respect to $u$.
\end{theorem}
\begin{proof}
See \cite[Theorems 4.2 and 4.3]{MR5008466}.
\end{proof}

The following result establishes suitable regularity properties of the adjoint variable $p$. To state this result, we introduce, for $y = \mathcal{S}(u)$, the set
\begin{equation}\label{def:set_E}
\mathcal{E} := \{t \in \mathcal{D}: y(t) \neq y_{t}\}.
\end{equation}
 
\begin{theorem}[local regularity]\label{thm:H^2-reg-p} 
Let $p \in W^{1,r}_{0}(\Omega)$ solve \eqref{eq:weak_adj_eq}, where $y = \mathcal{S}(u)$ and $u \in \mathbb{U}_{ad}$. If $\Omega$ is a convex polytope, then
$
p \in H^{2}(\Omega \setminus \cup_{t \in \mathcal{E}} \bar{B}_{t})\cap C^{0,1}(\bar{\Omega} \setminus \cup_{t \in \mathcal{E}}B_{t}).
$
Here, for each $t \in \mathcal{E}$, $B_{t} \subset \Omega$ is an open ball centered at $t$ with positive radius.
\end{theorem}
\begin{proof}
See \cite[Theorem 4.5]{MR5008466}.
\end{proof}
 
We now define the control-to-adjoint state map as follows:
\begin{equation}\label{eq:Phi}
\Phi: \mathcal{A} \rightarrow W^{1,r}_{0}(\Omega), \qquad  u \mapsto \Phi(u) = p,
\end{equation}
where $p$ is the unique solution to \eqref{eq:weak_adj_eq} and $r$ is as in \eqref{eq:r}. Recall that, as stated in Remark \ref{rmk:redef_A}, the open set $\mathcal{A} \subset L^2(\Omega)$ has been redefined. Moreover, given $u \in \mathcal{A}$ and $g \in L^{2}(\Omega)$, the directional derivative $\eta = \Phi'(u)g \in W_{0}^{1,r}(\Omega)$ is the unique solution to
\begin{multline}\label{eq:DPhi}
 \int_{\Omega}\nabla w \cdot \nabla \eta \mathrm{d}x +  \int_{\Omega}\left[ \frac{\partial a}{\partial y}(\cdot,y) + u \right]\eta w \mathrm{d}x \\
= \sum_{t \in \mathcal{D}} \langle z(t)\delta_{t},w \rangle - \int_{\Omega}\left[ \frac{\partial^{2} a}{\partial y^{2}}(\cdot,y)z + g \right]pw \mathrm{d}x
\quad \forall w \in W_0^{1,s}(\Omega),
\end{multline} 
where $r^{-1} + s^{-1} = 1$, $y = \mathcal{S}(u)$, and $z = \mathcal{S}'(u)g$; see \cite[Section 4.2.2]{MR5008466}.


\subsection{First order optimality conditions} 

Since \eqref{eq:weak_ocp}--\eqref{eq:weak_state_eq} is not convex, we study optimality conditions for locally optimal controls. We say that $\bar{u} \in \mathbb{U}_{ad}$ is a \emph{locally optimal control} for \eqref{eq:weak_ocp}--\eqref{eq:weak_state_eq} in the sense of $L^2(\Omega)$ if there exists $\varepsilon > 0$ such that
$
J(\mathcal{S}(\bar{u}),\bar{u}) \leq J(\mathcal{S}(u),u)
$
for all $u \in \mathbb{U}_{ad}$ such that $\| u - \bar{u}\|_{L^{2}(\Omega)} \leq \varepsilon$. We say that $\bar{u} \in \mathbb{U}_{ad}$ is a \emph{strict locally optimal control} in the sense of $L^2(\Omega)$ if there exists $\varepsilon>0$ such that
$
J(\mathcal{S}(\bar{u}),\bar{u}) < J(\mathcal{S}(u),u)
$
for all $u \in \mathbb{U}_{ad} \setminus \{\bar{u}\}$ such that $\| u - \bar{u}\|_{L^{2}(\Omega)} \leq \varepsilon$.

We now introduce the reduced cost functional $j: \mathcal{A} \rightarrow \mathbb{R}$, given by $j(u) := J(\mathcal{S}(u),u)$. Since $\mathcal{S}$ is of class $C^{2}$, it follows that $j$ is of class $C^{2}$. Moreover, for any $u \in \mathcal{A}$ and $g,g_{1},g_{2} \in L^{2}(\Omega)$, the first and second derivatives of $j$ are given by
\begin{align}
\label{eq:Dj}
j'(u)g & = \int_{\Omega}(\alpha u - yp)g \mathrm{d}x,
\\
\label{eq:DDj}
j''(u)(g_{1},g_{2}) & = \int_{\Omega}(\alpha g_{2} - z_{2}p - y\eta_{2})g_{1} \mathrm{d}x = \int_{\Omega}(\alpha g_{1} - z_{1}p - y\eta_{1})g_{2} \mathrm{d}x,
\end{align}
where $z_{i} = \mathcal{S}'(u)g_{i}$, $\eta_{i} = \Phi'(u)g_{i}$, and $i \in \{1,2\}$. Recall that $\mathcal{S}$ and $\Phi$ are defined in Sections \ref{sec:semilinear_eq} and \ref{sec:adjoint_eq}, respectively. These properties follow from \cite[Lemma 4.6]{MR5008466}.
 
We recall the following first order necessary optimality conditions.
\begin{theorem}[necessary first order optimality conditions]
Every locally optimal control $\bar{u} \in \mathbb{U}_{ad}$ satisfies the following variational inequality:
\begin{equation}\label{eq:char_var_ineq}
(\alpha \bar{u}-\bar{y}\bar{p},u-\bar{u})_{L^{2}(\Omega)} \geq 0 \quad \forall u \in \mathbb{U}_{ad},
\qquad
\bar{y} = \mathcal{S}(\bar{u}), \, \bar{p} = \Phi(\bar{u}).
\end{equation}
\end{theorem}
\begin{proof}
See \cite[Theorem 4.7]{MR5008466}.
\end{proof}

\subsection{Second order optimality conditions}

Let $(\bar{y},\bar{p},\bar{u}) \in (H_{0}^{1}(\Omega) \cap C^{0,\varsigma}(\bar{\Omega})) \times W_{0}^{1,r}(\Omega) \times \mathbb{U}_{ad}$ satisfy the first order optimality conditions \eqref{eq:weak_state_eq}, \eqref{eq:weak_adj_eq}, and  \eqref{eq:char_var_ineq}. Define $\bar{\mathfrak{p}}:= \alpha \bar{u} - \bar{y}\bar{p}$. By virtue of \eqref{eq:char_var_ineq}, it follows that, for a.e.~$x\in \Omega$,
\begin{equation}\label{eq:sign_con_j_prime}
     \bar{\mathfrak{p}}(x) \geq 0 \text{ if } \bar{u}(x) = \mathtt{a}, \qquad \bar{\mathfrak{p}}(x) \leq 0  \text{ if } \bar{u}(x) = \mathtt{b}, \qquad \bar{\mathfrak{p}}(x) = 0 \text{ if } \mathtt{a} < \bar{u}(x) < \mathtt{b}.
\end{equation}

We define the \emph{cone of critical directions}
\begin{equation}\label{eq:cone_critical}
C_{\bar{u}} := \left\lbrace g \in L^{2}(\Omega) \text{ satisfying \eqref{eq:sign_cond} and } \bar{\mathfrak{p}}(x) \neq 0 \implies g(x) = 0
\, \,
\mathrm{a.e.}~x \in \Omega \right \rbrace,
\end{equation}
where the condition \eqref{eq:sign_cond} is as follows:
\begin{equation}\label{eq:sign_cond}
    g(x) \geq 0 \text{ a.e. } x \in \Omega \text{ if } \bar{u}(x) = \mathtt{a}, \qquad g(x) \leq 0 \text{ a.e. } x \in \Omega \text{ if } \bar{u}(x) = \mathtt{b}.
\end{equation}

We now present necessary and sufficient second order optimality conditions.

\begin{theorem}[second order conditions]\label{thm:second_order_cond}
If $\bar{u} \in \mathbb{U}_{ad}$ is a locally optimal control, then $j''(\bar{u})g^{2} \geq 0$ for all $g \in C_{\bar{u}}$. Conversely, if $(\bar{y},\bar{p},\bar{u})$ satisfies \eqref{eq:weak_state_eq}, \eqref{eq:weak_adj_eq}, and \eqref{eq:char_var_ineq} and $j''(\bar{u})g^2>0$ for all $g \in C_{\bar{u}} \setminus \left\lbrace 0\right\rbrace$, then there exist $\mu, \sigma >0$ such that
\begin{equation}\label{eq:local_cuad_growth_j}
    j(u) \geq j(\bar{u}) + \frac{\mu}{2}\| u - \bar{u}\|^{2}_{L^{2}(\Omega)} \qquad \forall u \in \mathbb{U}_{ad}: \| u - \bar{u} \|_{L^{2}(\Omega)} \leq \sigma.
\end{equation}
In particular, $\bar{u} \in \mathbb{U}_{ad}$ is a locally optimal control for \eqref{eq:weak_ocp}--\eqref{eq:weak_state_eq} in the sense of $L^2(\Omega)$.
\end{theorem}
\begin{proof}
See \cite[Theorem 4.8]{MR5008466}.
\end{proof}

The following result will be important for deriving error bounds.

\begin{theorem}[equivalent second order optimality conditions]
If $(\bar{y},\bar{p},\bar{u})$ satisfies the first order optimality conditions \eqref{eq:weak_state_eq}, \eqref{eq:weak_adj_eq}, and \eqref{eq:char_var_ineq}, then
\begin{equation}\label{eq:equiv_cond}
j''(\bar{u})g^{2} > 0 \quad \forall g \in C_{\bar{u}} \setminus \left\lbrace 0 \right\rbrace
\iff
\exists \mu,\tau > 0: \, j''(\bar{u})g^{2} \geq \mu \|g\|^{2}_{L^{2}(\Omega)} \quad \forall g \in C^{\tau}_{\bar{u}},
 \end{equation}
 where $C_{\bar{u}}^{\tau}:=\left\lbrace g \in L^{2}(\Omega) \text{ satisfying \eqref{eq:sign_cond} and } |\bar{\mathfrak{p}}(x)|>\tau \implies g(x) = 0 \, \, \mathrm{a.e.}~x \in \Omega \right \rbrace$.
\end{theorem}
\begin{proof}
See \cite[Theorem 4.9]{MR5008466}.
\end{proof}


\section{Regularity results for a locally optimal control}\label{sec:loc_opt_control_reg}
In this section, we establish regularity properties of a locally optimal control, which are essential for deriving error estimates for our bilinear optimal control problem with pointwise tracking.

Define $\Pi_{[\mathtt{a},\mathtt{b}]}: L^{1}(\Omega) \rightarrow \mathbb{U}_{ad}$ by $\Pi_{[\mathtt{a},\mathtt{b}]}(v):= \min\{\mathtt{b},\max\{\mathtt{a},v\}\}$ a.e.~in $\Omega$. A locally optimal control for problem \eqref{eq:weak_ocp}--\eqref{eq:weak_state_eq} admits the characterization \cite[Section 4.6]{MR2583281}:
\begin{equation}\label{eq:u_proj_form}
\bar{u}(x) = \Pi_{[\mathtt{a},\mathtt{b}]}\left(\alpha^{-1}\bar{y}(x)\bar{p}(x) \right) \quad \text{for a.e. } x \in \Omega,
\qquad
\bar{y} = \mathcal{S}(\bar{u}),\, \bar{p} = \Phi(\bar{u}).
\end{equation}

%

\begin{theorem}[$H^1(\Omega)$-regularity of $\bar{u}$]\label{thm:H1_reg}
 If $\bar{u}$ is a locally optimal control of \eqref{eq:weak_ocp}--\eqref{eq:weak_state_eq}, then $\bar{u} \in H^1(\Omega)$.
\end{theorem}
\begin{proof}
 See \cite[Theorem 5.1]{MR5008466}.
\end{proof}

In \cite{MR5008466}, the global $C^{0,1}(\bar \Omega)$-regularity of a locally optimal control was established only in two dimensions. We show here that this result also holds in three dimensions, under the additional assumptions that $f - a(\cdot,0) \geq 0$ a.e.~in $\Omega$ and that $f \not \equiv a(\cdot,0)$. These assumptions guarantee that $\bar{y} > 0$ in $\Omega$, which is the key ingredient of the three-dimensional argument.

With these assumptions, we can now state the following regularity result for $\bar{u}$.

\begin{theorem}[global Lipschitz regularity of $\bar{u}$]\label{thm:C01_reg_baru}
Assume that $\Omega$ is a convex polytope and that $f,a(\cdot,0) \in L^{\eta}(\Omega)$ for some $\eta  > d$. Let $\bar{u}$ be a locally optimal control. Then, the following global regularity properties hold:
\begin{itemize}
\item[(i)]  If $d=2$ and the lower control bound $\mathtt{a}$ is strictly positive, then $\bar{u} \in C^{0,1}(\bar{\Omega})$.
\item[(ii)] If $d=3$ and $f-a(\cdot,0) \geq 0$ a.e.~in $\Omega$ and $f \not \equiv a(\cdot,0)$, then $\bar{u} \in C^{0,1}(\bar{\Omega})$.
\end{itemize}
\end{theorem}
\begin{proof}
We refer the reader to \cite[Theorem 5.2]{MR5008466} for a proof of item $(i)$; we prove item $(ii)$ below.

For the sake of simplicity, we assume that $\mathcal{D} = \mathcal{E} = \{t\}$. Recall that $\mathcal{E} = \{t \in \mathcal{D}: \bar{y}(t) \neq y_{t}\}$. The general case follows by applying the same argument in pairwise disjoint neighborhoods of the finitely many points of $\mathcal{E}$ and using Theorem \ref{thm:H^2-reg-p} away from these points. We use the projection formula in conjunction with \cite[Theorem A.1]{MR1786735} to conclude that $\nabla\bar{u} = \alpha^{-1}\nabla (\bar{y}\bar{p})\chi$ in the sense of distributions. Here, $\chi$ denotes the characteristic function of the set $\{x \in \Omega: \mathtt{a}< \alpha^{-1}\bar{y}(x)\bar{p}(x)<\mathtt{b}\}$. Using this identity, we control $\nabla \bar{u}$ in $L^{\infty}(\Omega)$: $\|\nabla \bar{u}\|_{L^{\infty}(\Omega)} = \alpha^{-1}  \|\nabla(\bar{y}\bar{p})\chi\|_{L^{\infty}(\Omega)}$. Since $\bar{p}$ is singular at $t \in \mathcal{E}$, we bound $\|\nabla(\bar{y}\bar{p})\chi\|_{L^{\infty}(\Omega)}$ both away from and near $t$:
$
\|\nabla \bar{u}\|_{L^{\infty}(\Omega)} \leq \alpha^{-1} \left( \|\nabla(\bar{y}\bar{p})\chi\|_{L^{\infty}(\Omega \setminus \bar{B}_{t}(\rho))} + \|\nabla (\bar{y}\bar{p})\chi\|_{L^{\infty}(B_{t}(\rho))} \right).
$
Here, $B_{t}(\rho)$ is an open ball centered at $t$ with radius $\rho>0$ such that $B_t(\rho)\Subset\Omega$.

From the proof of \cite[Theorem 5.2]{MR5008466}, it follows that $\|\nabla (\bar{y}\bar{p})\chi\|_{L^{\infty}(\Omega \setminus \bar{B}_{t}(\rho))}$ is finite. It thus suffices to control the term $\nabla (\bar{y}\bar{p})\chi$ in $L^{\infty}(B_{t}(\rho))$. The key observation is that $\bar{y}$ can be seen as the solution to a \emph{linear} PDE with a bounded, nonnegative reaction coefficient. Indeed, let us define
\[
 \mathfrak{q}(x) := \int_{0}^{1} \frac{\partial b}{\partial y}(x,s\bar{y}(x)) \, \mathrm{d}s + a_{0}(x) + \bar{u}(x)
 \quad
 \text{a.e.~} x \in \Omega,
\]
where
$
b:\Omega \times \mathbb{R} \rightarrow \mathbb{R}
$
is given by $b(x,y) := a(x,y) - a(x,0) - a_{0}(x)y$ (cf.~\cite[Theorem 2.4]{Casas_bilinear_1}). We immediately note that $b$ satisfies the following two properties. First, $b(x,0) = 0$ for a.e.~$x \in \Omega$. Second, $\partial b/\partial y (x,y) = \partial a / \partial y (x,y) - a_{0}(x) \geq 0 $ for a.e.~$x \in \Omega$ and for all $y \in \mathbb{R}$. Since $b(x,0) = 0$, the fundamental theorem of calculus gives
\begin{equation}
\label{eq:q_identity}
 \mathfrak{q}(x)\bar{y}(x) = b(x,\bar{y}(x)) + \left[a_{0}(x) + \bar{u}(x)\right]\bar{y}(x)
 \quad \text{a.e.~} x \in \Omega.
\end{equation}
The coefficient $\mathfrak{q}$ satisfies two properties. First, $\mathfrak{q} \geq 0$ a.e.~in $\Omega$: the integral defining $\mathfrak{q}$ is nonnegative by the second property of $b$, while  $a_0 + \bar{u} \geq 0$ a.e.~in $\Omega$, because $\bar{u} \in \mathbb{U}_{ad} \subset \mathcal{A}_0$. Second, $\mathfrak{q} \in L^{\infty}(\Omega)$: setting $\mathtt{m}:= \| \bar{y} \|_{L^{\infty}(\Omega)}$ (cf.~Theorem \ref{thm:well-posedness-semilinearPDE}), we have $|s\bar{y}(x)| \leq \mathtt{m}$ for every $s \in [0,1]$, whence assumption \ref{A3} yields $|\partial b/\partial y(x,s\bar{y}(x))| = |\partial a/\partial y(x,s\bar{y}(x)) - a_0(x)| \leq C_{a,\mathtt{m}} + \| a_0 \|_{L^{\infty}(\Omega)}$ for a.e.~$x \in \Omega$ and for every $s \in [0,1]$. Since $\bar{u} \in \mathbb{U}_{ad} \subset L^{\infty}(\Omega)$ and $a_0 \in L^{\infty}(\Omega)$, the claim $\mathfrak{q} \in L^{\infty}(\Omega)$ follows. The state $\bar{y} \in H^{1}_{0}(\Omega) \cap C^{0,\varsigma}(\bar{\Omega})$ can thus be seen as the unique weak solution to
\begin{equation}
 \Delta w - \mathfrak{q}w = -(f-a(\cdot,0)) \text{ in } \Omega, \quad w = 0 \text{ on } \partial \Omega.
 \label{eq:w}
\end{equation}
We now prove that $\bar{y} > 0$ in $\Omega$. Since $f-a(\cdot,0) \geq 0$ and $\mathfrak{q} \geq 0$ in $\Omega$, and $\bar{y} \in H_0^1(\Omega)$, an application of the weak minimum principle \cite[Theorem 8.1]{MR1814364} yields $\bar{y} \geq 0$ in $\Omega$. Since $\bar{y} \in C(\bar{\Omega})$, we thus have $\inf_{\Omega} \bar{y} = 0$. We now proceed by contradiction. Assume that the infimum is attained at an interior point of $\Omega$. An application of the strong minimum principle \cite[Theorem 8.19]{MR1814364} thus shows that $\bar{y}$ is constant in $\Omega$, and since $\bar{y} = 0$ on $\partial \Omega$, we have $\bar{y} \equiv 0$ in $\Omega$. In view of \eqref{eq:w}, we conclude that $f-a(\cdot,0) \equiv 0$, which contradicts the assumption $f-a(\cdot,0) \not \equiv 0$. We have proved that $\bar{y} > 0$ in $\Omega$.

Since $t\in\mathcal{D}\subset\Omega$, $\bar{y}(t) > 0$. Therefore, the arguments developed in \cite[Theorem 5.2, Case 1]{MR5008466} yield $\|\nabla (\bar{y}\bar{p})\chi\|_{L^{\infty}(B_{t}(\rho))} < \infty$. Consequently, $\bar{u} \in W^{1,\infty}(\Omega)$. Since $\Omega$ is convex, \cite[Theorem 4.1]{MR2177410} implies that $\bar{u} \in C^{0,1}(\bar{\Omega})$, which yields the desired result.
\end{proof}


\section{Finite element approximation}\label{sec:finite_element_disc}
We now introduce the discrete setting used throughout the finite element analysis \cite{CiarletBook,MR2050138,MR2373954}. Let $\mathscr{T}_{h} = \left\lbrace T \right\rbrace$ be a conforming partition, or mesh, of $\bar \Omega$ into closed simplices $T$ of size $h_{T}=\operatorname{diam}(T)$. Here, $h:= \max \{ h_{T}: T \in \mathscr{T}_{h} \}$. We denote by $\mathbb{T} = \left\lbrace \mathscr{T}_{h}\right\rbrace_{h>0}$ a collection of conforming, quasi-uniform refinements of an initial mesh $\mathscr{T}_{0}$. Given a mesh $\mathscr{T}_{h}$, we introduce
\begin{equation}\label{eq:Vh}
    \mathbb{V}_{h} := \left\lbrace v_{h} \in C(\bar{\Omega}): v_{h}|_{T} \in \mathbb{P}_{1}(T) \, \forall T \in \mathscr{T}_{h}\right\rbrace \cap H^{1}_{0}(\Omega).
\end{equation} 
\subsection{Finite element discretization of semilinear PDEs}
\label{sec:state_finite_element}
Let $\mathtt{u} \in \mathcal{A}_{0}$, and let $\mathtt{f} \in L^{2}(\Omega)$. Recall that the set $\mathcal{A}_{0}$ is defined in \eqref{eq:mathcal_A_0}. We define the Galerkin approximation of the solution $\mathtt{y}$ of problem \eqref{eq:weak_semi_bilinear_eq} by
\begin{equation}\label{eq:weak_semi_bilinear_eq_discr}
\mathtt{y}_{h} \in \mathbb{V}_{h}:
\quad
\int_{\Omega}\nabla \mathtt{y}_{h} \cdot \nabla v_{h}\mathrm{d}x
+
\int_{\Omega}a(\cdot,\mathtt{y}_{h})v_{h}\mathrm{d}x
+
\int_{\Omega}\mathtt{u}\mathtt{y}_{h}v_{h}\mathrm{d}x = \int_{\Omega}\mathtt{f}v_{h}\mathrm{d}x
\end{equation}
for all $v_{h} \in \mathbb{V}_{h}$. Here, $a$ satisfies \ref{A1}--\ref{A3}. An application of Brouwer's fixed point theorem yields the existence of a solution to \eqref{eq:weak_semi_bilinear_eq_discr} (see \cite[Lemma 8.2.1]{thesis_mateos}), while uniqueness follows from \ref{A2} and the fact that $\mathtt{u} \in \mathcal{A}_{0}$. Moreover, we have the stability bound $\|\nabla \mathtt{y}_{h}\|_{L^{2}(\Omega)} \lesssim \|\mathtt{f}-a(\cdot,0)\|_{H^{-1}(\Omega)}$, which follows by setting $v_h = \mathtt{y}_{h}$ in \eqref{eq:weak_semi_bilinear_eq_discr}. We refer the reader to \cite[Section 8.2]{thesis_mateos} and \cite[Section 7]{MR3586845} for further details.

\subsubsection{Error bounds} We present the following error bounds.

\begin{theorem}[error bounds]\label{thm:priori_state_estimates}
Assume that $\Omega$ is a convex polytope and that $a$ satisfies \ref{A1}--\ref{A3}. Let $\mathtt{u} \in \mathcal{A}_{0}$ and $\mathtt{f} \in L^{2}(\Omega)$, and let $\mathtt{y} \in H^{1}_{0}(\Omega) \cap H^{2}(\Omega)$ and $\mathtt{y}_{h} \in \mathbb{V}_{h}$ solve \eqref{eq:weak_semi_bilinear_eq} and \eqref{eq:weak_semi_bilinear_eq_discr}, respectively. If $h$ is sufficiently small, then
\begin{align}
\label{eq:state_priori_estimate_L^2}
\|\mathtt{y}-\mathtt{y}_{h}\|_{L^{2}(\Omega)} & \lesssim h^{2} \|\mathtt{f}-a(\cdot,0)\|_{L^{2}(\Omega)},
\\
\label{eq:state_priori_estimate_L^infty}
 \|\mathtt{y}-\mathtt{y}_{h}\|_{L^{\infty}(\Omega)} & \lesssim h^{2-\frac{d}{2}} \|\mathtt{f}-a(\cdot,0)\|_{L^{2}(\Omega)}.
\end{align}
Let $\Omega_{1} \Subset \Omega_{0} \Subset \Omega$, with $\Omega_{0}$ smooth. If $\mathtt{f},a(\cdot,0) \in L^{\infty}(\Omega_{0})$ and $\mathtt{u} \in \mathbb{U}_{ad}$, then
\begin{equation}\label{eq:state_local_L^infty_est}
\| \mathtt{y} - \mathtt{y}_{h} \|_{L^{\infty}(\Omega_{1})} \lesssim h^{2}|\log h|^{2}.
\end{equation}
In all estimates, the hidden constants depend on $\mathtt{u}$ but are independent of $h$; if, in addition, $\mathtt{u} \in \mathbb{U}_{ad}$, they can be chosen uniformly with respect to $\mathtt{u}$.
\end{theorem}
\begin{proof}
The bounds \eqref{eq:state_priori_estimate_L^2} and \eqref{eq:state_priori_estimate_L^infty} follow essentially from the arguments in the proofs of \cite[Lemma 4]{MR2009948} and \cite[Theorem 1]{MR2009948}, respectively. Finally, \eqref{eq:state_local_L^infty_est} follows from slight modifications of the arguments in the proof of \cite[Theorem 4.1]{MR4438718}.
\end{proof}
%
\subsubsection{The discrete control-to-state map}\label{sec:discrete_Sh_mapping}

We now examine differentiability properties of the map $\mathtt{u} \mapsto \mathtt{y}_{h}$, where $\mathtt{y}_{h}$ denotes the solution of \eqref{eq:weak_semi_bilinear_eq_discr}. These properties are necessary to derive first order optimality conditions for the proposed schemes.

\begin{theorem}[differentiability of $\mathtt{u} \mapsto \mathtt{y}_{h}$]\label{thm:diff_dicrete_S}
Let $a: \Omega \times \mathbb{R} \rightarrow \mathbb{R}$ satisfy \ref{A1}--\ref{A3}. Then, for every $h>0$, there exists an open set $\mathcal{A}_{h} \subset L^{2}(\Omega)$ such that $\mathcal{A}_{0} \subset \mathcal{A}_{h} \subset \mathcal{A}$ and, for all $\mathtt{u} \in \mathcal{A}_{h}$, problem \eqref{eq:weak_semi_bilinear_eq_discr}  has a unique solution $\mathtt{y}_{h} \in \mathbb{V}_{h}$. Moreover, there exists a map $\mathcal{S}_{h}: \mathcal{A}_{h} \rightarrow \mathbb{V}_{h}$ of class $C^{2}$ that satisfies the following properties:
\begin{itemize}[leftmargin=*,nosep]
\item[(i)] $\mathcal{S}_{h}(\mathtt{u}) = \mathtt{y}_{h} \in \mathbb{V}_{h}$, where $\mathtt{y}_{h}$ is the unique solution to \eqref{eq:weak_semi_bilinear_eq_discr},
\item[(ii)] for every $\mathtt{g} \in L^{2}(\Omega)$,
$\mathtt{z}_{h} = \mathcal{S}'_{h}(\mathtt{u})\mathtt{g} \in \mathbb{V}_{h}$ is the unique solution to the problem
\begin{equation}
\label{eq:z_h_equation}
\int_{\Omega} \nabla \mathtt{z}_{h} \cdot \nabla v_{h} \mathrm{d}x + \int_{\Omega}\left[ \frac{\partial a}{\partial y}(\cdot,\mathtt{y}_{h}) + \mathtt{u} \right]\mathtt{z}_{h}v_{h}\mathrm{d}x = -\int_{\Omega} \mathtt{g} \mathtt{y}_{h}v_{h} \mathrm{d}x
\quad
\forall v_{h} \in \mathbb{V}_{h}.
\end{equation}
\end{itemize}
\end{theorem}
\begin{proof}
The proof follows from arguments similar to those developed in the proof of \cite[Theorem 4.1]{MR4956345}. In our setting, the open set $\mathcal{A}_{h}$ is defined as follows:
\begin{equation}\label{def:A_h}
\mathcal{A}_{h} \subset L^{2}(\Omega),
\quad
\mathcal{A}_{h}:= \bigcup \left \{ B_{\varepsilon^{h}_{\bar{\mathtt{u}}}}(\bar{\mathtt{u}}): \bar{\mathtt{u}} \in \mathcal{A}_{0} \right \} \cap \mathcal{A}, \quad 0 < \varepsilon^{h}_{\bar{\mathtt{u}}} < C^{-2}_{4 \hookrightarrow 2}.
\end{equation}
Here, $B_{\varepsilon^{h}_{\bar{\mathtt{u}}}}(\bar{\mathtt{u}})$ is the open ball in $L^{2}(\Omega)$ centered at $\bar{\mathtt{u}} \in \mathcal{A}_{0}$ of radius $\varepsilon^{h}_{\bar{\mathtt{u}}}$ and $C_{4 \hookrightarrow 2}$ is the best constant in the Sobolev embedding $H^{1}_{0}(\Omega) \hookrightarrow L^{4}(\Omega)$. Finally, the well-posedness of problem \eqref{eq:z_h_equation} follows from the arguments presented in \cite[Remark 3.6]{MR5008466}.
\end{proof}
%
\subsection{Finite element discretization of the adjoint equation}

Let $u \in \mathcal{A}_h$ and $y = \mathcal{S}(u)$. We define the Galerkin approximation of the solution $p$ to the adjoint equation \eqref{eq:weak_adj_eq} by
\begin{equation}\label{eq:weak_adj_eq_discr}
q_{h} \in \mathbb{V}_{h}: \int_{\Omega}\nabla w_{h} \cdot \nabla q_{h}\mathrm{d}x + \int_{\Omega}\left[\frac{\partial a}{\partial y}(\cdot,y)+ u\right]q_{h}w_{h}\mathrm{d}x = \sum_{t \in \mathcal{D}}\langle (y(t) - y_{t})\delta_{t},w_{h} \rangle
\end{equation}
for all $w_{h} \in \mathbb{V}_{h}$. If $u \in \mathcal{A}_0$, it is clear that \eqref{eq:weak_adj_eq_discr} has a unique solution. If $u \in \mathcal{A}_h \setminus \mathcal{A}_0$, the definition of $\mathcal{A}_h$ yields the existence of $\bar{u} \in \mathcal{A}_0$ such that $\| u - \bar{u} \|_{L^2(\Omega)} < C_{4 \hookrightarrow 2}^{-2}$. Standard computations, as in \cite[Remark 3.6]{MR5008466}, show that \eqref{eq:weak_adj_eq_discr} also has a unique solution in this case.

\subsubsection{Error bounds}
We present the following error bounds.

\begin{theorem}[error bounds]\label{thm:error_estimate_p}
Assume that $\Omega$ is a convex polytope and that $a$ satisfies \ref{A1}--\ref{A3}. Let $u \in \mathcal{A}_{0}$, $p = \Phi(u)$, and let $q_{h} \in \mathbb{V}_h$ solve \eqref{eq:weak_adj_eq_discr}. Then,
\begin{equation}\label{eq:adj_priori_estimate_L^2}
\|p-q_{h}\|_{L^{2}(\Omega)}
\lesssim h^{2-\frac{d}{2}} \left[\|f-a(\cdot,0)\|_{L^{2}(\Omega)} + \sum_{t \in \mathcal{D}}|y_{t}| \right].
\end{equation}
If, in addition, $u \in \mathbb{U}_{ad}$, then there exists $h_{\sharp} > 0$ such that
\begin{equation}\label{eq:adj_priori_estimate_L^1}
 \|p-q_{h}\|_{L^{1}(\Omega)} \lesssim h^{2} |\log h|^{2} \quad \forall h < h_{\sharp}.
\end{equation}
The hidden constants depend on $u$ but are independent of $h$; if, in addition, $u \in \mathbb{U}_{ad}$, they can be chosen uniformly with respect to $u$.
\end{theorem}
\begin{proof}
Let $c_0$ be given by $c_{0}(x)= \partial a / \partial y (x,y(x)) + u(x)$ for a.e.~$x \in \Omega$, where $y = \mathcal{S}(u)$. By \ref{A2} and $u \in \mathcal{A}_0$, $c_{0}\geq a_{0} + u \geq 0$ a.e.~in $\Omega$. The bounds \eqref{eq:adj_priori_estimate_L^2} and \eqref{eq:adj_priori_estimate_L^1} follow from slight modifications of the proofs of \cite[Theorem 3]{MR0812624} and \cite[Lemma 5.3]{MR3973329}, respectively. Note that \cite[Theorem 3]{MR0812624} extends to $c_{0} \in L^{2}(\Omega)$. These arguments use $c_{0}$ only through $c_{0} \geq 0$ and $c_{0} \in L^{2}(\Omega)$ for \eqref{eq:adj_priori_estimate_L^2}, and additionally $c_{0} \in L^{\infty}(\Omega)$ for \eqref{eq:adj_priori_estimate_L^1}. The estimates thus hold for any coefficient with these properties.
\end{proof}

The following result is essential for deriving error estimates for the discretization of \eqref{eq:weak_ocp}--\eqref{eq:weak_state_eq} using the variational discretization approach (see Section \ref{sec:error_est_SD}).

\begin{theorem}[error bound]
Assume that $\Omega$ is a convex polytope and that $a$ satisfies \ref{A1}--\ref{A3}. Let $u \in \mathbb{U}_{ad}$, $p = \Phi(u)$, and $q_{h}$ be the solution of \eqref{eq:weak_adj_eq_discr}. Then, there exists $h_{\circ} > 0$ such that
\begin{equation}\label{eq:error_est_adjoint_out_balls}
\|p-q_{h}\|_{L^{2}(\Omega \setminus \cup_{t \in \mathcal{D}}\bar{B}_{t})} \lesssim h^{2} |\log h| \quad \forall h < h_{\circ},
\end{equation}
where the hidden constant is independent of $u \in \mathbb{U}_{ad}$ and $h$. Here, $B_{t} \subset \Omega$ denotes a suitable open ball centered at $t \in \mathcal{D}$ of strictly positive radius.
\end{theorem}
\begin{proof}
The proof follows from an adaptation of \cite[Lemma 5.5]{MR3973329}. Let $0 < \varrho_{1} < \varrho_{0} < \varrho$ be such that $B_t(\varrho) \Subset \Omega$ and $\bar{B}_t(\varrho) \cap \mathcal{D} = \{ t \}$ for every $t \in \mathcal{D}$, with $\varrho$ small enough such that $\{B_{t}(\varrho)\}_{t \in \mathcal{D}}$ are pairwise disjoint. Set $B = \cup \{ B_t(\varrho): t \in \mathcal{D} \}$ and $\Omega_i:= \cup \{ B_t(\varrho_i): t \in \mathcal{D} \}$, for $i \in \{0,1\}$. Note that $\mathcal{D} \subset \Omega_1 \Subset \Omega_0 \Subset B \Subset \Omega$. Let $\varphi \in H_0^1(\Omega)$ solve the associated dual problem with right-hand side $\mathfrak{f}:=(1-\chi_{B})(p-q_{h})$, where $\chi_B$ denotes the characteristic function of $B$, and let $\varphi_h$ be its Galerkin approximation. Testing with $p$ (via a density argument) and $q_h$, respectively, we obtain
\[
 \|p-q_{h}\|^{2}_{L^{2}(\Omega\setminus B)} = \sum_{t \in \mathcal{D}}\langle
(y(t)-y_{t})\delta_{t}, \varphi-\varphi_{h}\rangle.
\]
Since $\mathfrak{f}$ vanishes on each $B_t(\varrho_0)$, the local bound of \cite[Lemma 4.4(ii)]{MR3973329} applies on $B_t(\varrho_1)$. As a result, $\|\varphi-\varphi_{h}\|_{L^{\infty}(B_t(\varrho_1))} \lesssim h^{2}|\log h| \| \mathfrak{f} \|_{L^{2}(\Omega)}$. Consequently, $ \|p-q_{h}\|_{L^{2}(\Omega\setminus B)} \lesssim h^{2}|\log h| \sum_{t \in \mathcal{D}} |y(t)-y_{t}|$. Note that
$\sum_{t\in\mathcal{D}}|y(t)-y_t|\lesssim 1$ uniformly in $u$ and $h$.
\end{proof}
%
\subsection{Auxiliary error bounds}
Let $u_{h} \in \mathbb{U}_{ad}$ and set $y_{h} = \mathcal{S}_{h}(u_{h})$, where $\mathcal{S}_h$ is the discrete control-to-state map defined in Theorem \ref{thm:diff_dicrete_S}. We define the discrete adjoint state $p_{h} \in \mathbb{V}_{h}$ as the solution of
\begin{equation}
\label{eq:discr_adj_eq}
\int_{\Omega} \nabla w_{h} \cdot \nabla p_{h}\mathrm{d}x
+
\int_{\Omega}\left[\frac{\partial a}{\partial y}(\cdot,y_{h}) + u_{h}\right]p_{h}w_{h}\mathrm{d}x
=
\sum_{t \in \mathcal{D}}\langle (y_{h}(t) - y_{t})\delta_{t},w_{h} \rangle
\end{equation}
for all $w_{h} \in \mathbb{V}_{h}$. We next derive an auxiliary error estimate that will be important for deriving a priori error bounds for both finite element schemes we propose for \eqref{eq:weak_ocp}--\eqref{eq:weak_state_eq}. The proof of this estimate relies on the auxiliary state $y(u_h)$, defined as follows. For $u_h\in\mathbb{U}_{ad}$, $y(u_h) := \mathcal{S}(u_h) \in H_0^1(\Omega) \cap L^{\infty}(\Omega)$; that is, $y(u_h)$ solves
\begin{equation}\label{eq:yuh}
\int_{\Omega} \nabla y(u_h) \cdot \nabla v \mathrm{d}x
+
\int_{\Omega} a(\cdot,y(u_h))v \mathrm{d}x
+
\int_{\Omega} u_h y(u_h) v \mathrm{d}x = \int_{\Omega} f v \mathrm{d}x
\end{equation}
for all $v \in H_0^1(\Omega)$. The existence and uniqueness of $y(u_h)$ follow from \cite[Theorem 3.5]{MR5008466}. More importantly, for any sequence $\{u_{h}\}_{h>0} \subset \mathbb{U}_{ad}$, we have
\begin{equation}
\label{eq:estimates_uniform_in_h}
 \| \nabla y(u_h) \|_{L^2(\Omega)} \leq \| f - a(\cdot,0) \|_{H^{-1}(\Omega)},
 \quad
 \| y(u_h) \|_{L^{\infty}(\Omega)} \leq C_{\infty} \| f - a(\cdot,0) \|_{L^q(\Omega)},
\end{equation}
where $q > d/2$ and $C_{\infty}>0$ is independent of $h$. Using these bounds, assumption \ref{A3}, and the fact that $\{u_h \}_{h>0}$ is uniformly bounded in $L^2(\Omega)$, we obtain
\begin{equation}\label{eq:H2_bound_yuh}
|y(u_{h})|_{H^{2}(\Omega)} \leq C\|f-a(\cdot,0)\|_{L^{2}(\Omega)},
\end{equation}
where $C>0$ is independent of $h$.

\begin{theorem}[auxiliary error bound]\label{thm:aux_error_estimate_adj_state}
Assume that $\Omega$ is a convex polytope and that $a$ satisfies \ref{A1}--\ref{A3}. Let $u \in \mathbb{U}_{ad}$, and let $\{u_{h}\}_{h>0} \subset \mathbb{U}_{ad}$. Let $p = \Phi(u)$, and let $p_{h} \in \mathbb{V}_{h}$ be the solution of \eqref{eq:discr_adj_eq}. Then, there exists $h_{\circledast} > 0$ such that
\begin{equation}\label{eq:fully_est_p_L2}
\|p - p_{h}\|_{L^{2}(\Omega)} \lesssim \|u - u_{h}\|_{L^{2}(\Omega)} + h^{2 - \tfrac{d}{2}} \quad \forall h<h_{\circledast},
\end{equation}
where the hidden constant is independent of $h$ and uniform with respect to $u_h \in \mathbb{U}_{ad}$.
\end{theorem}
\begin{proof}
We begin with the triangle inequality $\|p-p_{h}\|_{L^{2}(\Omega)} \leq \|p-\mathsf{p}(u_h)\|_{L^{2}(\Omega)} + \|\mathsf{p}(u_h)-p_{h}\|_{L^{2}(\Omega)}$, where $\mathsf{p}(u_h) \in W^{1,r}_{0}(\Omega)$ is the unique solution of
\begin{equation}
\label{eq:rhouh}
\int_{\Omega} \nabla w \cdot \nabla \mathsf{p}(u_h) \mathrm{d}x + \int_{\Omega} \left[ \frac{\partial a}{\partial y}(\cdot,y_{h}) + u_{h}\right]\mathsf{p}(u_h) w \mathrm{d}x = \sum_{t \in \mathcal{D}}\langle (y_{h}(t)-y_{t})\delta_{t},w \rangle
\end{equation}
for all $w \in W^{1,s}_{0}(\Omega)$. Here, $r$ is as in \eqref{eq:r}. We first control $\|\mathsf{p}(u_h)-p_{h}\|_{L^{2}(\Omega)}$. Since $p_h\in\mathbb{V}_h$ is precisely the Galerkin approximation of $\mathsf{p}(u_h)$, the bound \eqref{eq:adj_priori_estimate_L^2} and the observation following its proof yield
\begin{equation}\label{eq:aux_est_adj_error_1}
\|\mathsf{p}(u_h) - p_{h}\|_{L^{2}(\Omega)} \lesssim h^{2-\tfrac{d}{2}} \sum_{t \in \mathcal{D}}|y_{h}(t)-y_{t}|
\lesssim h^{2-\tfrac{d}{2}}\left( \|y_{h}\|_{L^{\infty}(\Omega)} + \sum_{t \in \mathcal{D}}|y_{t}|\right).
\end{equation}
Note that $\|y_{h} \|_{L^{\infty}(\Omega)}$ is uniformly bounded with respect to $h$. In fact, since $y_{h} \in \mathbb{V}_{h}$ is precisely the finite element approximation of $y(u_h)$, defined as the solution of \eqref{eq:yuh}, the error estimate \eqref{eq:state_priori_estimate_L^infty} and the uniform $L^{\infty}(\Omega)$ bound in \eqref{eq:estimates_uniform_in_h} yield
\begin{equation}
 \|y_{h}\|_{L^{\infty}(\Omega)} \leq \|y(u_h) - y_{h}\|_{L^{\infty}(\Omega)}  + \| y(u_h) \|_{L^{\infty}(\Omega)} \lesssim h^{2-\tfrac{d}{2}} + 1.
 \label{eq:yh_uniformly_bounded}
\end{equation}

We now estimate $\|p-\mathsf{p}(u_h)\|_{L^{2}(\Omega)}$. To this end, we define $\psi:=p-\mathsf{p}(u_h)$. Subtracting the equations satisfied by $p$ and $\mathsf{p}(u_h)$ yields
\begin{multline*}
\psi \in W^{1,r}_{0}(\Omega):
\quad
\int_{\Omega} \nabla w \cdot \nabla \psi \mathrm{d}x
+
\int_{\Omega}\left[ \frac{\partial a}{\partial y}(\cdot,y) + u\right]\psi w\mathrm{d}x
=  \sum_{t \in \mathcal{D}}  \langle (y(t)-y_{h}(t))\delta_{t},w \rangle
\\
- \int_{\Omega} \left[ \frac{\partial a}{\partial y}(\cdot,y) - \frac{\partial a}{\partial y}(\cdot,y_{h})\right]\mathsf{p}(u_h)w \mathrm{d}x - \int_{\Omega} (u-u_{h}) \mathsf{p}(u_h)w\mathrm{d}x \quad \forall w \in W^{1,s}_{0}(\Omega).
\end{multline*}
Since $u \in \mathbb{U}_{ad} \subset \mathcal{A}_{0}$, \cite[Theorem 4.2]{MR5008466} guarantees the well-posedness of this problem and yields the stability estimate
\begin{multline}\label{eq:aux_est_adj_error_2}
\|\nabla \psi \|_{L^{r}(\Omega)} \lesssim \sum_{t \in \mathcal{D}}|y(t)-y_{h}(t)| + \left\|\left[ \frac{\partial a}{\partial y}(\cdot,y)-\frac{\partial a}{\partial y}(\cdot,y_{h})\right]\mathsf{p}(u_h) \right\|_{W^{-1,r}(\Omega)}  \\
+ \|(u-u_{h})\mathsf{p}(u_h)\|_{W^{-1,r}(\Omega)} =: \textbf{I}_{h} + \textbf{II}_{h} + \textbf{III}_{h}.
\end{multline}
We estimate the three terms on the right-hand side separately. For $\mathbf{I}_{h}$, we have
\begin{equation}
\label{eq:I_h}
 \textbf{I}_{h} \lesssim \|y-y(u_h)\|_{L^{\infty}(\Omega)} + \|y(u_h)-y_{h}\|_{L^{\infty}(\Omega)} \lesssim  \|u-u_{h}\|_{L^{2}(\Omega)} + h^{2-\tfrac{d}{2}}.
\end{equation}
Here, we have used an $L^{\infty}(\Omega)$ stability estimate for $y-y(u_h)$, obtained by considering the linear problem satisfied by $y-y(u_h)$ and applying \cite[Theorem B.2]{MR1786735}. We have also used that $y_{h}$ is the finite element approximation of $y(u_h)$ and the error estimate \eqref{eq:state_priori_estimate_L^infty}. We now bound $\textbf{II}_{h}$. To do this, we use the stability estimate for the problem that $\mathsf{p}(u_h)$ solves together with the uniform boundedness of $\|y_h\|_{L^\infty(\Omega)}$ and obtain
\[
\|\nabla \mathsf{p}(u_h)\|_{L^r(\Omega)}
\lesssim
\sum_{t\in\mathcal D}|y_h(t)-y_t|
\lesssim 1,
\]
uniformly with respect to $h$. The $L^\infty(\Omega)$ boundedness of $y$, the uniform $L^\infty(\Omega)$ boundedness of $y_h$, the local Lipschitz continuity of $\partial a/\partial y$ with respect to the second variable, and the embeddings $W^{1,r}_{0}(\Omega) \hookrightarrow L^{2}(\Omega)$ and $W^{1,s}_{0}(\Omega) \hookrightarrow C(\bar{\Omega})$ then yield
$
\textbf{II}_{h} \lesssim \|y - y_{h} \|_{L^{\infty}(\Omega)}\|\nabla \mathsf{p}(u_h)\|_{L^{r}(\Omega)} \lesssim \|u-u_{h}\|_{L^{2}(\Omega)} + h^{2-d/2}.
$
The last estimate follows from the bounds used for $\textbf{I}_{h}$. A bound for $\mathbf{III}_{h}$ is analogous and simpler. Substituting the bounds for $\mathbf{I}_{h}, \mathbf{II}_{h}$, and $\mathbf{III}_{h}$ into \eqref{eq:aux_est_adj_error_2} yields
\begin{equation}\label{eq:aux_est_adj_error_3}
\|p - \mathsf{p}(u_h) \|_{L^{2}(\Omega)}
=
\|\psi\|_{L^{2}(\Omega)}
\lesssim
\|\nabla\psi\|_{L^{r}(\Omega)} \lesssim \|u-u_{h}\|_{L^{2}(\Omega)} + h^{2-\tfrac{d}{2}}.
\end{equation}
Combining \eqref{eq:aux_est_adj_error_1}, \eqref{eq:yh_uniformly_bounded}, and \eqref{eq:aux_est_adj_error_3} with the triangle inequality yields \eqref{eq:fully_est_p_L2}.
\end{proof}

\subsection{Discretization of the control problem}\label{sec:disc_ocp}
We propose two discretization schemes for approximating solutions to the optimal control problem \eqref{eq:weak_ocp}--\eqref{eq:weak_state_eq}: a fully discrete scheme, where the admissible control set is discretized using piecewise constant functions, and a semidiscrete scheme, based on the so-called variational discretization approach \cite{MR2122182}, where the set $\mathbb{U}_{ad}$ is not discretized.

Before describing these schemes, we introduce the discrete reduced cost functional 
$
j_{h}: \mathcal{A}_{h} \rightarrow \mathbb{R},
$
defined by
$
j_{h}(\mathtt{u}):=J(\mathcal{S}_{h}(\mathtt{u}),\mathtt{u}).
$
Since $\mathcal{S}_{h}(\mathtt{u}) \in \mathbb{V}_{h} \subset C(\bar{\Omega})$ for every $\mathtt{u} \in \mathcal{A}_{h}$ (cf.~Theorem \ref{thm:diff_dicrete_S}), $j_h$ is well-defined.

We now derive differentiability results for $j_{h}$.

\begin{lemma}[differentiability of $j_{h}$]\label{lem:jh_diff} For every $h > 0$, $j_{h}: \mathcal{A}_{h} \rightarrow \mathbb{R}$ is of class $C^{2}$. Moreover, for every $\mathtt{u} \in \mathcal{A}_{h}$ and $\mathtt{g} \in L^2(\Omega)$, the following identity holds:
\begin{equation}
\label{eq:Djh}
j'_{h}(\mathtt{u})\mathtt{g} = \int_{\Omega}( \alpha \mathtt{u} - \mathtt{y}_{h} \phi_{h}) \mathtt{g} \mathrm{d}x.
\end{equation}
Here, $\mathtt{y}_{h} = \mathcal{S}_{h}(\mathtt{u}) \in \mathbb{V}_h$ and $\phi_{h} \in \mathbb{V}_h$ is the unique solution of
\begin{equation}
\label{eq:phi_h}
\int_{\Omega} \nabla \phi_{h} \cdot \nabla w_{h} \mathrm{d}x
+
\int_{\Omega}\left[\frac{\partial a}{\partial y}(\cdot,\mathtt{y}_{h})+ \mathtt{u} \right]\phi_{h}w_h\mathrm{d}x
=
\sum_{t \in \mathcal{D}}\langle (\mathtt{y}_{h}(t) - y_{t})\delta_{t},w_{h} \rangle
\end{equation}
for all $w_{h} \in \mathbb{V}_{h}$, i.e., $\phi_{h}$ solves \eqref{eq:discr_adj_eq} with $u_{h} = \mathtt{u}$ and
$y_{h} = \mathtt{y}_{h}$.
\end{lemma}
\begin{proof}
Since $\mathcal{S}_{h}$ is of class $C^{2}$ (cf.~Theorem \ref{thm:diff_dicrete_S}), the chain rule implies that $j_{h}$ is also of class $C^{2}$. We now derive \eqref{eq:Djh}. Let $\mathtt{g} \in L^{2}(\Omega)$. Basic computations show that
$
j'_{h}(\mathtt{u})\mathtt{g} = \sum_{t \in \mathcal{D}}(\mathtt{y}_{h}(t)-y_{t}) \mathtt{z}_{h}(t) + (\alpha \mathtt{u},\mathtt{g})_ {L^{2}(\Omega)},
$
where $\mathtt{z}_{h} = \mathcal{S}'_{h}(\mathtt{u})\mathtt{g}$; see item $(ii)$ in Theorem \ref{thm:diff_dicrete_S}. Setting $v_{h} = \phi_{h}$ in \eqref{eq:z_h_equation} and $w_{h} = \mathtt{z}_{h}$ in \eqref{eq:phi_h}, we obtain
$
\sum_{t \in \mathcal{D}}(\mathtt{y}_{h}(t)-y_{t}) \mathtt{z}_{h}(t) = - \int_{\Omega} \mathtt{g}  \mathtt{y}_{h}  \phi_{h} \mathrm{d}x.
$
Combining these identities yields \eqref{eq:Djh}.
\end{proof}

\subsubsection{A fully discrete scheme}\label{sec:FD_scheme}
We propose the following fully discrete approximation of problem \eqref{eq:weak_ocp}--\eqref{eq:weak_state_eq}: Minimize $J(y_{h},u_{h})$ subject to
\begin{equation}\label{eq:fully_eq}
 \int_{\Omega} \nabla y_{h} \cdot \nabla v_{h} \mathrm{d}x
 +
 \int_{\Omega} a(\cdot,y_{h})v_{h} \mathrm{d}x + \int_{\Omega}u_{h}y_{h}v_{h} \mathrm{d}x = \int_{\Omega} fv_{h} \mathrm{d}x \quad \forall v_{h} \in \mathbb{V}_{h},
\end{equation}
and the control constraint $u_{h} \in \mathbb{U}_{ad,h}$, where $\mathbb{U}_{ad,h}:=\mathbb{U}_{h} \cap \mathbb{U}_{ad}$ and $\mathbb{U}_{h} = \{ v_{h} \in L^{\infty}(\Omega): v_{h}|_{T} \in \mathbb{P}_{0}(T) ~ \forall T \in \mathscr{T}_{h} \}$. Note that $y_{h} = \mathcal{S}_{h}(u_{h})$. We recall that $\mathcal{S}_{h}$ and $\mathbb{V}_{h}$ are defined in Theorem \ref{thm:diff_dicrete_S} and \eqref{eq:Vh}, respectively.

Since $\mathbb{U}_{ad,h} \subset \mathbb{U}_{ad} \subset \mathcal{A}_0 \subset \mathcal{A}_h$, $\mathcal{S}_h$ is well-defined on $\mathbb{U}_{ad,h}$ and the existence of a discrete solution follows from the compactness of $\mathbb{U}_{ad,h}$ and the continuity of $j_h$. By Lemma \ref{lem:jh_diff}, the first order optimality conditions for the fully discrete optimal control problem read as follows: If $\bar{u}_{h} \in \mathbb{U}_{ad,h}$ is a locally optimal control, then
\begin{equation}\label{eq:fully_var_ineq}
j'_{h}(\bar{u}_{h})(u_{h}-\bar{u}_{h}) = \int_{\Omega} (\alpha \bar{u}_{h} - \bar{y}_{h}\bar{p}_{h})(u_{h}-\bar{u}_{h}) \mathrm{d}x \geq 0 \quad \forall u_{h} \in \mathbb{U}_{ad,h}.
\end{equation}
Here, $\bar{y}_{h} = \mathcal{S}_{h}(\bar{u}_{h})$ and $\bar{p}_{h} \in \mathbb{V}_{h}$ solves \eqref{eq:discr_adj_eq} with $y_{h} = \bar{y}_{h}$ and $u_{h} = \bar{u}_{h}$.

\subsubsection{A semidiscrete scheme}
\label{sec:SD_scheme}

In this scheme, the state equation is discretized by finite elements, while the control space is left undiscretized. The resulting scheme is as follows \cite{MR2122182}: Minimize $J(y_{h},\mathfrak{u})$ subject to
\begin{equation}\label{eq:state_eq_semidisc}
\int_{\Omega} \nabla y_{h} \cdot \nabla v_{h} \mathrm{d}x + \int_{\Omega}a(\cdot,y_{h})v_{h} \mathrm{d}x + \int_{\Omega}\mathfrak{u}y_{h}v_{h} \mathrm{d}x = \int_{\Omega}fv_{h} \mathrm{d}x \quad \forall v_{h} \in \mathbb{V}_{h},
\end{equation}
and the control constraint $\mathfrak{u} \in \mathbb{U}_{ad}$. The existence of a discrete solution follows from standard arguments. In view of the characterization of $j'_{h}$ in Lemma \ref{lem:jh_diff}, the first order conditions are as follows: If $\bar{\mathfrak{u}} \in \mathbb{U}_{ad}$ is a locally optimal control, then
\begin{equation}\label{eq:first_opt_cond_semidiscrete}
j'_{h}(\bar{\mathfrak{u}})(u-\bar{\mathfrak{u}}) = \int_{\Omega}(\alpha \bar{\mathfrak{u}}-\bar{y}_{h}\bar{p}_{h})(u-\bar{\mathfrak{u}})\mathrm{d}x \geq 0 \quad \forall u \in \mathbb{U}_{ad},
\end{equation}
where $\bar{y}_{h} = \mathcal{S}_{h}(\bar{\mathfrak{u}})$ and $\bar{p}_{h} \in \mathbb{V}_{h}$ solves \eqref{eq:discr_adj_eq} with $y_{h} = \bar{y}_{h}$ and $u_{h} =  \bar{\mathfrak{u}}$.
By the inequality \eqref{eq:first_opt_cond_semidiscrete}, $\bar{\mathfrak{u}}$ admits the projection characterization \cite[Section 4.6]{MR2583281}:
\begin{equation}\label{eq:disc_opt_ct_SD}
\bar{\mathfrak{u}}(x) = \Pi_{[\mathtt{a},\mathtt{b}]}(\alpha^{-1}\bar{y}_{h}(x)\bar{p}_{h}(x))  \quad \text{for a.e.}~x \in \Omega.
\end{equation}
Since $\bar{\mathfrak{u}}$ depends implicitly on $h$, we will use the notation $\bar{\mathfrak{u}}_{h}$ from now on.

\section{Convergence}\label{sec:conv_discretizations}
In this section, we analyze convergence properties of the fully discrete and semidiscrete schemes, beginning with an auxiliary result.

\begin{theorem}[error bounds and convergence]\label{thm:aux_error_estimate_state}
Assume that $\Omega$ is a convex polytope and that \ref{A1}--\ref{A3} hold. Let $u \in \mathbb{U}_{ad}$, and let $\{ u_{h} \}_{h >0}$ be such that $u_{h} \in \mathbb{U}_{ad,h} \subset \mathbb{U}_{ad}$ for every $h>0$. Let $y = \mathcal{S}(u)$ be the solution to \eqref{eq:weak_state_eq} and $y_{h} = \mathcal{S}_{h}(u_{h})$ be the solution to \eqref{eq:fully_eq}. If $h$ is sufficiently small, then
\begin{equation}\label{eq:fully_state_estimates}
\|\nabla(y-y_{h})\|_{L^{2}(\Omega)} \lesssim h + \|u-u_{h}\|_{L^{2}(\Omega)}, \quad \|y-y_{h}\|_{L^{\infty}(\Omega)} \lesssim h^{2-\frac{d}{2}} + \|u-u_{h}\|_{L^{2}(\Omega)}.  
\end{equation}
The hidden constants are independent of $h$ and uniform with respect to $u, u_{h} \in \mathbb{U}_{ad}$. Moreover, if $u_{h} \rightharpoonup u$ in $L^{2}(\Omega)$ as $h \rightarrow 0$, then $y_{h} \rightarrow y$ in $H^{1}_{0}(\Omega) \cap C(\bar{\Omega})$ as $h \rightarrow 0$ and $j(u) \leq \liminf_{h \rightarrow 0}j_{h}(u_{h})$.
\end{theorem}
\begin{proof}
We present a proof of the first error bound. The $L^{\infty}(\Omega)$ error bound follows directly from the estimates used to derive \eqref{eq:I_h} in the proof of Theorem \ref{thm:aux_error_estimate_adj_state}. We begin with a simple application of the triangle inequality
\begin{equation}\label{eq:triangle_ineq_y-yh}
\|\nabla( y - y_{h})\|_{L^{2}(\Omega)} \leq \|\nabla(y-y(u_h))\|_{L^{2}(\Omega)} + \|\nabla(y(u_h)- y_{h})\|_{L^{2}(\Omega)},
\end{equation}
where $y(u_h) \in H^{1}_{0}(\Omega) \cap L^{\infty}(\Omega)$ is the unique solution of \eqref{eq:yuh}. Since $y_{h}$ is the finite element approximation of $y(u_h)$, standard arguments, together with the uniform $L^{\infty}(\Omega)$-boundedness of $\{ y(u_h) \}_{h>0}$ and $\{ y_h \}_{h>0}$ (see \eqref{eq:estimates_uniform_in_h} and \eqref{eq:yh_uniformly_bounded}, respectively), and the uniform $L^2(\Omega)$-boundedness of $\{u_h\}_{h>0}\subset\mathbb{U}_{ad}$, yield
\begin{equation}\label{eq:yuh-yh}
 \|\nabla(y(u_h) - y_{h})\|_{L^{2}(\Omega)} \lesssim h |y(u_h)|_{H^2(\Omega)} \lesssim h \| f - a(\cdot,0)\|_{L^2(\Omega)},
\end{equation}
where the last bound follows from \eqref{eq:H2_bound_yuh}. From the linear problem satisfied by $y-y(u_{h})$, we obtain $\|\nabla(y-y(u_{h}))\|_{L^{2}(\Omega)} \lesssim \|(u-u_{h}) y(u_h)\|_{L^2(\Omega)} \lesssim \|u-u_{h}\|_{L^{2}(\Omega)}$, where the last bound follows from the uniform $L^{\infty}(\Omega)$-boundedness of $\{ y(u_h) \}_{h>0}$. Combining this bound with \eqref{eq:yuh-yh} and \eqref{eq:triangle_ineq_y-yh} yields the first bound in \eqref{eq:fully_state_estimates}.

We now prove the last two assertions. By \eqref{eq:triangle_ineq_y-yh} and \eqref{eq:yuh-yh}, we have
\begin{equation}
\label{eq:noname}
\|\nabla(y-y_h)\|_{L^2(\Omega)}
\lesssim
\|\nabla(y-y(u_h))\|_{L^2(\Omega)}
+
h\|f-a(\cdot,0)\|_{L^2(\Omega)}.
\end{equation}
We next show that $y(u_h) \rightharpoonup y$ in $H^2(\Omega)$ and, in fact, that $y(u_h) \to y$ in $H_0^1(\Omega) \cap C(\bar\Omega)$. By \eqref{eq:estimates_uniform_in_h} and \eqref{eq:H2_bound_yuh}, $\{y(u_h)\}_{h>0}$ is uniformly bounded in $H^2(\Omega)\cap H_0^1(\Omega)$. Let $\{y(u_{h_k})\}_{k\in\mathbb{N}}$ be an arbitrary subsequence. Up to a further subsequence, not relabeled, there exists $y_\star\in H^2(\Omega)\cap H_0^1(\Omega)$ such that $y(u_{h_k})\rightharpoonup y_\star$ in $H^2(\Omega)$. Since $H^2(\Omega)\cap H_0^1(\Omega)$ is compactly embedded into $H_0^1(\Omega)\cap C(\bar\Omega)$, passing, if necessary, to a further subsequence, again not relabeled, we also have $y(u_{h_k}) \to y_{\star}$ in $H_0^1(\Omega) \cap C(\bar \Omega)$ as $k \uparrow \infty$. Passing to the limit in \eqref{eq:yuh}, using that $u_{h_k} \rightharpoonup u$ in $L^2(\Omega)$ and that $y(u_{h_k}) \to y_{\star}$ in $H_0^1(\Omega) \cap C(\bar \Omega)$ as $k \uparrow \infty$, we conclude that $y_{\star}$ solves \eqref{eq:weak_state_eq}. By uniqueness, $y_{\star} = y$. Since the original subsequence was arbitrary, we conclude that $y(u_h)\rightharpoonup y$ in $H^2(\Omega)$ and $y(u_h)\to y$ in $H_0^1(\Omega)\cap C(\bar\Omega)$ as $h \to 0$. Combining this result with \eqref{eq:noname} yields $y_{h} \rightarrow y$ in $H^{1}_{0}(\Omega)$ as $h \rightarrow 0$. The desired convergence in $C(\bar\Omega)$ follows from the triangle inequality, the convergence $y(u_h)\to y$ in $C(\bar\Omega)$, and the estimate $\|y(u_h)-y_h\|_{L^\infty(\Omega)}
\lesssim h^{2-d/2}.$
Finally, the strong convergence $y_h\to y$ in $C(\bar\Omega)$ and the weak lower semicontinuity of $\|\cdot\|_{L^2(\Omega)}^2$ yield $j(u)\leq\liminf_{h\to0}j_h(u_h)$.
\end{proof}

\subsection{The fully discrete scheme: convergence of discretizations}

We show that a sequence $\{\bar{u}_{h}\}_{h>0}$ of globally optimal controls of the fully discrete scheme admits a subsequence converging to a globally optimal control of \eqref{eq:weak_ocp}--\eqref{eq:weak_state_eq}.

\begin{theorem}[convergence of global solutions]\label{thm:conv_global_sol}
Assume that $\Omega$ is a convex polytope and that \ref{A1}--\ref{A3} hold. Let $\{\bar{u}_{h}\}_{h>0}$ be such that $\bar{u}_{h} \in \mathbb{U}_{ad,h}$ is a globally optimal control of the fully discrete scheme. Then, there exists a nonrelabeled subsequence of $\{\bar{u}_{h}\}_{h>0}$ such that $\bar{u}_{h} \rightharpoonup \bar{u}$ in $L^{2}(\Omega)$ as $h \rightarrow 0$ and $\bar{u}$ is a globally optimal control of  \eqref{eq:weak_ocp}--\eqref{eq:weak_state_eq}. In addition, we have the following properties:
\begin{equation}
\label{eq:conv_properties_global_solutions}
\lim_{h \rightarrow 0}\|\bar{u}-\bar{u}_{h}\|_{L^{2}(\Omega)} = 0, \quad \lim_{h \rightarrow 0}j_{h}(\bar{u}_{h}) = j(\bar{u}).
\end{equation}     
\end{theorem}

\begin{proof}
Since $\{\bar{u}_{h}\}_{h>0}$ is uniformly bounded in $L^{2}(\Omega)$, there exists a nonrelabeled subsequence $\{\bar{u}_{h}\}_{h>0}$ such that $\bar{u}_{h} \rightharpoonup \bar{u}$ in $L^{2}(\Omega)$ as $h \rightarrow 0$. Arguing as in \cite[Theorem 39]{MR3586845}, and using the convergence properties of Theorem \ref{thm:aux_error_estimate_state} together with the regularity result of Theorem \ref{thm:H1_reg}, we conclude that $\bar{u}$ is a globally optimal control of \eqref{eq:weak_ocp}--\eqref{eq:weak_state_eq} and that the properties in \eqref{eq:conv_properties_global_solutions} hold. For brevity, we omit the details.
\end{proof}

\begin{theorem}[convergence of local solutions]\label{thm:conv_local_sol}
Assume that $\Omega$ is a convex polytope and that \ref{A1}--\ref{A3} hold. Let $\bar{u} \in \mathbb{U}_{ad}$ be a strict locally optimal control of \eqref{eq:weak_ocp}--\eqref{eq:weak_state_eq}. Then, there exist $h_{\square}>0$ and a sequence of locally optimal controls $\{\bar{u}_{h}\}_{0<h\leq h_{\square}}$ of the fully discrete scheme such that \eqref{eq:conv_properties_global_solutions} holds.
\end{theorem}
\begin{proof}
The result follows by arguing as in the proof of \cite[Theorem 5.3]{MR4438718}. For brevity, we omit the details.
\end{proof}

\subsection{The semidiscrete scheme: convergence of discretizations}\label{sec:semidiscete_convergence}
We now establish the analogues of Theorems \ref{thm:conv_global_sol}
and \ref{thm:conv_local_sol} for the semidiscrete scheme.
%
%
\begin{theorem}[convergence of the semidiscrete scheme]\label{thm:conv_semidisc}
Assume that $\Omega$ is a convex polytope and that \ref{A1}--\ref{A3} hold.
\begin{itemize}[leftmargin=*,nosep]
\item[(i)] Let $\{\bar{\mathfrak{u}}_{h}\}_{h>0}$ be such that $\bar{\mathfrak{u}}_{h} \in \mathbb{U}_{ad}$ is a globally optimal control of the semidiscrete scheme. Then, there exists a nonrelabeled subsequence $\{\bar{\mathfrak{u}}_{h}\}_{h >0}$ such that $\bar{\mathfrak{u}}_{h} \rightharpoonup \bar{u}$ in $L^{2}(\Omega)$ as $h \rightarrow 0$, $\bar{u}$ is a globally optimal control of \eqref{eq:weak_ocp}--\eqref{eq:weak_state_eq}, and \eqref{eq:conv_properties_global_solutions} holds.
\item[(ii)] Let $\bar{u} \in \mathbb{U}_{ad}$ be a strict locally optimal control of \eqref{eq:weak_ocp}--\eqref{eq:weak_state_eq}. Then, there exist $h_{\Delta}>0$ and a sequence of locally optimal controls $\{\bar{\mathfrak{u}}_{h}\}_{0<h\leq h_{\Delta}}$ of the semidiscrete scheme such that \eqref{eq:conv_properties_global_solutions} holds.
\end{itemize}
\end{theorem}
\begin{proof}
The proofs follow from adapting \cite[Theorem 39]{MR3586845} and \cite[Theorem 5.3]{MR4438718}. Note that the proof of Theorem \ref{thm:aux_error_estimate_state} applies verbatim to $\{\mathfrak{\bar{u}}_{h}\}_{h>0} \subset \mathbb{U}_{ad}$.
\end{proof}

\section{Error estimates}\label{sec:error_estimates}
We derive error estimates for the fully discrete and semidiscrete schemes introduced in Sections \ref{sec:FD_scheme} and \ref{sec:SD_scheme}, respectively. Throughout this section, we assume that $\Omega$ is a convex polytope.

\subsection{Error estimates for the fully discrete scheme}\label{sec:error_estimates_fully} 
 Let $\{\bar{u}_{h}\}_{h>0}$ be a sequence of locally optimal controls for the fully discrete scheme such that $\bar{u}_h \in \mathbb{U}_{ad,h}$ and $\bar{u}_{h} \rightarrow \bar{u}$ in $L^{2}(\Omega)$ as $h \rightarrow 0$, where $\bar{u}$ is a locally optimal control of \eqref{eq:weak_ocp}--\eqref{eq:weak_state_eq} satisfying \eqref{eq:equiv_cond}; see Theorems \ref{thm:conv_global_sol} and \ref{thm:conv_local_sol}.
Our goal is to derive the error estimate
\begin{equation}\label{eq:error_estimate_control}
 \|\bar{u}-\bar{u}_{h}\|_{L^{2}(\Omega)} \lesssim h |\log h| \quad \forall h < h_{\ddagger}, \quad h_{\ddagger} > 0.
\end{equation}
We proceed by contradiction, following \cite[Lemma 4.2]{MR2350349}. Assume that \eqref{eq:error_estimate_control} does not hold. Then, there exists a sequence $\{h_{k}\}_{k \in \mathbb{N}} \subset \mathbb{R}^{+}$ with $h_k \to 0$ as $k \to \infty$ such that
\begin{equation}\label{eq:contradict_assumption}
\lim_{k \to \infty}\|\bar{u}-\bar{u}_{h_{k}}\|_{L^{2}(\Omega)} = 0,
\qquad
\lim_{k \to \infty}\dfrac{\|\bar{u}-\bar{u}_{h_{k}}\|_{L^{2}(\Omega)}}{h_{k}|\log h_{k}|}  = +\infty.
\end{equation}
The first limit follows from $\bar{u}_{h}\to\bar{u}$ in $L^{2}(\Omega)$. We now establish the following bound.

\begin{lemma}[auxiliary bound]\label{lem:aux_result}
Assume that \ref{A1}--\ref{A3} hold and that $\bar{u} \in \mathbb{U}_{ad}$ satisfies the second order optimality condition \eqref{eq:equiv_cond}. Let $\{h_k \}_{k\in\mathbb{N}} \subset \mathbb{R}^{+}$ be a sequence satisfying \eqref{eq:contradict_assumption}. Then, there exists $k_{\dagger} \in \mathbb{N}$ such that
\begin{equation}\label{eq:aux_control_error_bound}
\mathfrak{C} \|\bar{u}-\bar{u}_{h_{k}}\|^{2}_{L^{2}(\Omega)} \leq [j'(\bar{u}_{h_{k}})-j'(\bar{u})](\bar{u}_{{h}_{k}}-\bar{u})
\qquad
\forall k \geq k_{\dagger},
\end{equation} 
where $\mathfrak{C} = 2^{-1}\min\{\mu,\alpha\}$, $\alpha$ is the regularization parameter, and $\mu>0$ is as in \eqref{eq:equiv_cond}.
\end{lemma}
\begin{proof}
We omit the index $k$ to simplify the notation and write $h$ and $\bar{u}_{h}$ in place of $h_k$ and $\bar{u}_{h_k}$, respectively. Applying the mean value theorem, we obtain
\begin{equation}\label{eq:mvt_auxiliary_error_control}
[j'(\bar{u}_{h})-j'(\bar{u})](\bar{u}_{h}-\bar{u})
=
j''(\hat{u}_{h})(\bar{u}_{h}-\bar{u})^{2},
\quad
\hat{u}_{h}
=
\bar{u}+\theta_{h}(\bar{u}_{h}-\bar{u}),
\quad
\theta_{h}\in(0,1).
\end{equation}
Define $g_{h} := (\bar{u}_{h}-\bar{u})/\|\bar{u}_{h}-\bar{u}\|_{L^{2}(\Omega)}$. Then, $\|g_{h} \|_{L^2(\Omega)} = 1$. Thus, up to a nonrelabeled subsequence, $g_{h} \rightharpoonup g$ in $L^{2}(\Omega)$ as $h \rightarrow 0$. We now proceed in two steps.

\emph{Step 1.} We prove that $g \in C_{\bar{u}}$. Since $\bar{u}_{h} \in \mathbb{U}_{ad,h}$, $g_{h}$ satisfies the sign conditions \eqref{eq:sign_cond}. This fact, together with the weak convergence $g_{h} \rightharpoonup g$ in $L^{2}(\Omega)$ as $h \rightarrow 0$, implies that $g$ satisfies \eqref{eq:sign_cond}. It remains to prove that $g(x)=0$ for a.e.~$x\in\Omega$ such that $\bar{\mathfrak{p}}(x)\neq0$, where $\bar{\mathfrak{p}} = \alpha\bar{u} - \bar{y}\bar{p}$ (cf.~\eqref{eq:cone_critical}). For this purpose, we introduce $\bar{\mathfrak{p}}_{h} := \alpha \bar{u}_{h} - \bar{y}_{h} \bar{p}_{h}$, where $\bar{y}_{h} = \mathcal{S}_{h}(\bar{u}_{h})$ and $\bar{p}_{h} \in \mathbb{V}_{h}$ is the unique solution of \eqref{eq:discr_adj_eq} with $y_h=\bar y_h$ and $u_h=\bar u_h$. We now estimate $\|\bar{\mathfrak{p}} - \bar{\mathfrak{p}}_{h}\|_{L^{2}(\Omega)}$. Using the second bound in \eqref{eq:stability_weak_problem}, Theorems \ref{thm:aux_error_estimate_adj_state} and \ref{thm:aux_error_estimate_state}, and $\|\bar{u}-\bar{u}_{h}\|_{L^{2}(\Omega)} \rightarrow 0$ as $h \rightarrow 0$, we obtain
\begin{multline*}
 \|\bar{\mathfrak{p}}-\bar{\mathfrak{p}}_{h}\|_{L^{2}(\Omega)} \leq \alpha\|\bar{u}-\bar{u}_{h}\|_{L^{2}(\Omega)} +  \|\bar{y}\|_{L^{\infty}(\Omega)} \|\bar{p}-\bar{p}_{h}\|_{L^{2}(\Omega)} \\
+ \|\bar{p}_{h}\|_{L^{2}(\Omega)} \|\bar{y}-\bar{y}_{h}\|_{L^{\infty}(\Omega)} \lesssim \|\bar{u}-\bar{u}_{h}\|_{L^{2}(\Omega)} + h^{2 - \tfrac{d}{2}} \rightarrow 0, \quad h \rightarrow 0.
\end{multline*}
Hence, $\bar{\mathfrak{p}}_{h} \rightarrow \bar{\mathfrak{p}}$ in $L^{2}(\Omega)$ as $h \rightarrow 0$. Combining this with $g_h \rightharpoonup g$ in $L^{2}(\Omega)$, we obtain
\[
\int_{\Omega} \bar{\mathfrak{p}} g \, \mathrm{d}x = \lim_{h \rightarrow 0} \frac{1}{\|\bar{u}_{h}-\bar{u}\|_{L^{2}(\Omega)}}\left[ \int_{\Omega}\bar{\mathfrak{p}}_{h}(\bar{u}_{h}-\Pi_{h}(\bar{u}))\mathrm{d}x + \int_{\Omega}\bar{\mathfrak{p}}_{h}(\Pi_{h}(\bar{u})-\bar{u})\mathrm{d}x\right].
\]
Here, $\Pi_{h}$ denotes the $L^{2}$-orthogonal projection onto $\mathbb{U}_{h}$. Since $\Pi_{h}(\bar{u}) \in \mathbb{U}_{ad,h}$, the discrete variational inequality \eqref{eq:fully_var_ineq} yields $(\bar{\mathfrak{p}}_{h},\bar{u}_{h} - \Pi_{h}(\bar{u}))_{L^{2}(\Omega)} \leq 0$. Moreover, $\bar{\mathfrak{p}}_{h} \rightarrow \bar{\mathfrak{p}}$ in $L^{2}(\Omega)$ as $h \rightarrow 0$ implies the existence of $\mathfrak{h}>0$ such that $\{\bar{\mathfrak{p}}_{h}\}_{0<h\leq\mathfrak{h}}$ is uniformly bounded in $L^2(\Omega)$. Since $\bar{u} \in H^{1}(\Omega)$ by Theorem \ref{thm:H1_reg}, we have $\|\Pi_{h}(\bar{u})-\bar{u}\|_{L^{2}(\Omega)} \lesssim h|\bar{u}|_{H^{1}(\Omega)}$. These estimates and properties, in view of \eqref{eq:contradict_assumption}, yield
\begin{equation*}
	\begin{split}
\int_{\Omega} \bar{\mathfrak{p}} g \mathrm{d}x
& \leq
 \lim_{h \rightarrow 0}\frac{1}{\|\bar{u}_{h}-\bar{u}\|_{L^{2}(\Omega)}} \left| \int_{\Omega} \bar{\mathfrak{p}}_{h}(\Pi_{h}(\bar{u}) - \bar{u})\mathrm{d}x \right|
\lesssim
\lim_{h \rightarrow 0} \frac{\|\Pi_{h}(\bar{u})-\bar{u}\|_{L^{2}(\Omega)}}{\|\bar{u}_{h}-\bar{u}\|_{L^{2}(\Omega)}}
\\
 & \leq C\lim_{h \rightarrow 0}
 \frac{h}{\|\bar{u}_{h}-\bar{u}\|_{L^{2}(\Omega)}}
 \leq C
 \lim_{h \rightarrow 0}
 \frac{h |\log h | }{\|\bar{u}_{h}-\bar{u}\|_{L^{2}(\Omega)}}
 = 0.
\end{split}
\end{equation*}
Since $g$ satisfies \eqref{eq:sign_cond}, \eqref{eq:sign_con_j_prime} yields $\bar{\mathfrak{p}}(x)g(x) \geq 0$ for a.e.~$x\in\Omega$, so that $|\bar{\mathfrak{p}}(x)g(x)| = \bar{\mathfrak{p}}(x)g(x)$. The previous estimate yields $\int_{\Omega}|\bar{\mathfrak{p}}g|\,\mathrm{d}x \leq 0$, whence $\bar{\mathfrak{p}}(x)g(x) = 0$ for a.e. $x \in \Omega$; that is, $g(x)=0$ for a.e.~$x\in\Omega$ such that $\bar{\mathfrak p}(x)\neq0$. This shows that $g \in C_{\bar{u}}$.

\emph{Step 2.} We now prove the required lower bound for $j''(\hat{u}_{h})g_h^2$. By \eqref{eq:DDj},
\begin{equation}
j''(\hat{u}_{h})g^{2}_{h} = \alpha \|g_{h}\|^{2}_{L^{2}(\Omega)}
-
\int_{\Omega} \left[ z(\hat{u}_{h})p(\hat{u}_{h})+ y(\hat{u}_{h}) \eta(\hat{u}_{h}) \right] g_{h} \,\mathrm{d}x.
\end{equation}
Here, $y(\hat{u}_{h}) = \mathcal{S}(\hat{u}_{h})$, $z(\hat{u}_{h}) = \mathcal{S}'(\hat{u}_{h})g_{h}$, $p(\hat{u}_{h}) = \Phi(\hat{u}_{h})$, and $\eta(\hat{u}_{h}) = \Phi'(\hat{u}_{h})g_{h}$. Arguing as in \cite[Theorem 4.8]{MR5008466}, we obtain $y(\hat{u}_{h}) \rightarrow \bar{y}$ and $z(\hat{u}_{h}) \rightarrow \bar{z}$ in $H^{1}_{0}(\Omega) \cap C^{0,\varsigma}(\bar{\Omega})$, and $p(\hat{u}_{h}) \rightarrow \bar{p}$ and $\eta(\hat{u}_{h}) \rightarrow \bar{\eta}$ in $W^{1,r}_{0}(\Omega)$ as $h \rightarrow 0$, where $\bar{y} = \mathcal{S}(\bar{u})$, $\bar{z} = \mathcal{S}'(\bar{u})g$, $\bar{p} = \Phi(\bar{u})$, and $\bar{\eta} = \Phi'(\bar{u})g$. Since $\|g_{h}\|_{L^{2}(\Omega)}=1$, the first term equals $\alpha$. The above convergences imply $z(\hat{u}_{h})p(\hat{u}_{h}) \rightarrow \bar{z}\bar{p}$ and $y(\hat{u}_{h})\eta(\hat{u}_{h}) \rightarrow \bar{y}\bar{\eta}$ in $L^{2}(\Omega)$. Combining these convergences with $g_{h} \rightharpoonup g$ in $L^{2}(\Omega)$, \eqref{eq:DDj}, and \eqref{eq:equiv_cond}, we obtain
\begin{equation*}
\lim_{h \rightarrow 0}j''(\hat{u}_{h})g^{2}_{h} = \alpha - (\bar{z}\bar{p}+\bar{y}\bar{\eta},g)_{L^{2}(\Omega)} = \alpha + j''(\bar{u})g^{2} - \alpha \|g\|^{2}_{L^{2}(\Omega)} \geq \alpha + (\mu - \alpha) \|g\|^{2}_{L^{2}(\Omega)}.
\end{equation*}
Since $\|g\|_{L^{2}(\Omega)} \leq 1$, $\lim_{h \rightarrow 0}j''(\hat{u}_{h})g^{2}_{h} \geq \min\{\mu,\alpha\}>0$ for the extracted subsequence. Since every subsequence admits a further subsequence to which the above argument applies, we conclude that $j''(\hat{u}_{h})g^{2}_{h} > 2^{-1} \min\{\mu,\alpha\}$ for $h$ sufficiently small. Combining this bound with the definition of $g_h$ and  \eqref{eq:mvt_auxiliary_error_control} and returning to the original notation, we deduce that there exists $k_{\dagger}\in\mathbb{N}$ such that \eqref{eq:aux_control_error_bound} holds for every $k\geq k_{\dagger}$.
\end{proof}

We now provide auxiliary bounds that will be used for both schemes.

\begin{lemma}[auxiliary bounds for $j'$ and $j'_{h}$]\label{lem:estimates_for_j'_j'h} Assume that \ref{A1}--\ref{A3} hold and that $f, a(\cdot,0) \in L^{\infty}(\Omega)$. Let $u_{1},u_{2} \in \mathbb{U}_{ad}$. Then, there exists $h_{\odot}>0$, independent of $u_1,u_2\in\mathbb U_{ad}$, such that for every $h<h_{\odot}$,
\begin{align}
 & |j'(u_{1})g - j'_{h}(u_{1})g| \lesssim h^{2}|\log h|^{2} \|g\|_{L^{\infty}(\Omega)} \quad \forall g \in L^{\infty}(\Omega),
 \label{eq:aux_j-jh_estimate_1} \\
 & |j'_{h}(u_{1})g - j'_{h}(u_{2})g| \lesssim \|u_{1}-u_{2}\|_{L^{2}(\Omega)} \|g\|_{L^{2}(\Omega)} \quad \forall g \in L^{2}(\Omega).
 \label{eq:aux_j-jh_estimate_2}
\end{align}
The hidden constants are independent of $h$ and of $u_{1},u_{2} \in \mathbb{U}_{ad}$.
\end{lemma}
\begin{proof}
We first prove \eqref{eq:aux_j-jh_estimate_1}. Let $g \in L^{\infty}(\Omega)$. Using \eqref{eq:Dj} and \eqref{eq:Djh}, we write $j'(u_{1})g = (\alpha u_{1} - y_{1}p_{1},g)_{L^{2}(\Omega)}$ and $j'_{h}(u_{1})g = (\alpha u_{1} - y_{1,h}\phi_{1,h}, g)_{L^{2}(\Omega)}$. Here, $y_{1} = \mathcal{S}(u_{1})$, $p_{1} = \Phi(u_{1})$,  $y_{1,h} = \mathcal{S}_{h}(u_{1})$, and $\phi_{1,h} \in \mathbb{V}_{h}$ is the unique solution of \eqref{eq:phi_h} with $\mathtt{y}_{h}$ and $\mathtt{u}$ replaced by $y_{1,h}$ and $u_{1}$, respectively. Let $\phi_{1} \in W^{1,r}_{0}(\Omega)$ be the solution of
\begin{equation*}
\int_{\Omega} \nabla w \cdot \nabla \phi_{1} \mathrm{d}x
+
\int_{\Omega} \left[ \frac{\partial a}{\partial y}(\cdot,y_{1,h}) + u_{1}\right]\phi_{1}w \mathrm{d}x = \sum_{t \in \mathcal{D}}\langle (y_{1,h}(t)-y_{t})\delta_{t},w \rangle
\end{equation*}
for all $w \in W^{1,s}_{0}(\Omega)$. This problem is well-posed because the reaction coefficient is nonnegative and uniformly bounded in $L^{\infty}(\Omega)$ \cite[Theorem 4.2]{MR5008466}; the corresponding stability bound holds with a constant independent of $h$ and $u_{1} \in \mathbb{U}_{ad}$. We then write
\begin{equation}
\begin{aligned}
\label{eq:decomposition}
j'(u_{1})g - j'_{h}(u_{1})g & =  (y_{1,h}[\phi_{1,h}-\phi_{1}],g)_{L^{2}(\Omega)} + (y_{1,h}[\phi_{1} - p_{1}],g)_{L^{2}(\Omega)} \\
& + (p_{1}[y_{1,h}-y_{1}],g)_{L^{2}(\Omega)}  =: \mathfrak{M}_{h} + \mathfrak{D}_{h} + \mathfrak{E}_{h}.
\end{aligned}
\end{equation}
We begin by estimating $\mathfrak{E}_{h}$. Let $\Omega_1$, $\Lambda_{1}$, and $\Omega_{0}$ be smooth domains such that $\mathcal{D} \subset \Omega_{1} \Subset \Lambda_{1} \Subset \Omega_{0} \Subset \Omega$. Invoking the estimates \eqref{eq:state_priori_estimate_L^2} and \eqref{eq:state_local_L^infty_est}, the embedding $W^{1,r}_{0}(\Omega) \hookrightarrow L^{2}(\Omega)$, and the definitions of $y_{1}$ and $y_{1,h}$, we obtain
\begin{align*}
\|p_{1}[y_{1,h}-y_{1}]\|^{2}_{L^{2}(\Omega)} & \leq \|p_{1}\|^{2}_{L^{2}(\Lambda_{1})} \|y_{1,h}-y_{1}\|^{2}_{L^{\infty}(\Lambda_{1})} + \|p_{1}\|^{2}_{L^{\infty}(\Omega \setminus \Lambda_{1})} \|y_{1,h}-y_{1}\|^{2}_{L^{2}(\Omega)}  \\
& \lesssim h^{4}|\log h|^{4} \|\nabla p_{1}\|^{2}_{L^{r}(\Omega)} + h^{4} \|p_{1}\|^{2}_{L^{\infty}(\Omega \setminus \Lambda_{1})} \|f-a(\cdot,0)\|^{2}_{L^{2}(\Omega)}.  
\end{align*} 
Since $p_{1}$ is uniformly bounded with respect to $u_1\in\mathbb U_{ad}$ in $W^{1,r}_{0}(\Omega)$ (cf.~Theorem \ref{thm:wp_adjoint}) and, by the proof of Theorem \ref{thm:H^2-reg-p}, in $H^2(\Omega\setminus\Lambda_{1}) \hookrightarrow L^{\infty}(\Omega\setminus\Lambda_{1})$, we conclude that $\|p_{1}[y_{1,h}-y_{1}]\|_{L^{2}(\Omega)} \lesssim h^{2}|\log h|^{2}$. Consequently,
\begin{equation}\label{eq:est_1h}
|\mathfrak{E}_{h}| \leq \|p_{1}[y_{1,h}-y_{1}]\|_{L^{2}(\Omega)} \|g\|_{L^{2}(\Omega)} \lesssim h^{2}|\log h|^{2}\|g\|_{L^{2}(\Omega)}.
\end{equation}
We next estimate $\mathfrak{D}_{h}$. To this end, we define $\xi:=p_{1}-\phi_{1} \in W^{1,r}_{0}(\Omega)$. Subtracting the equations satisfied by $p_{1}$ and $\phi_{1}$, we obtain
\begin{multline*}
\int_{\Omega} \nabla w \cdot \nabla \xi \mathrm{d}x + \int_{\Omega} \left[ \frac{\partial a}{\partial y}(\cdot,y_{1}) + u_{1}\right]\xi w \mathrm{d}x  = \sum_{t \in \mathcal{D}}\langle(y_{1}(t)-y_{1,h}(t))\delta_{t}, w \rangle \\
- \int_{\Omega}\left[ \frac{\partial a}{\partial y}(\cdot,y_{1}) - \frac{\partial a}{\partial y}(\cdot,y_{1,h})\right]\phi_{1}w\mathrm{d}x \quad \forall w \in W^{1,s}_{0}(\Omega). \nonumber
\end{multline*}
Using the stability estimate for the adjoint problem, the definition of the $W^{-1,r}(\Omega)$-norm, the bounds \eqref{eq:state_local_L^infty_est} and \eqref{eq:state_priori_estimate_L^2}, and the uniform boundedness of $\phi_1$ in $W^{1,r}_{0}(\Omega)$, we obtain $\|\nabla \xi\|_{L^{r}(\Omega)} \lesssim h^{2}|\log h|^{2}$. This estimate, the Sobolev embedding $W^{1,r}_{0}(\Omega) \hookrightarrow L^{2}(\Omega)$, and the uniform bound $\| y_{1,h}\|_{L^{\infty}(\Omega)} \lesssim 1$, obtained as in \eqref{eq:yh_uniformly_bounded}, yield
\begin{equation}\label{eq:est_2h}
|\mathfrak{D}_{h}| \leq \| y_{1,h} \|_{L^{\infty}(\Omega)} \|\xi\|_{L^{2}(\Omega)} \| g\|_{L^{2}(\Omega)} \lesssim h^{2}|\log h|^{2} \|g\|_{L^{2}(\Omega)}.
\end{equation}
Finally, since $u_{1} \in \mathbb{U}_{ad}$ and $\phi_{1,h} \in \mathbb{V}_{h}$ is the finite element approximation of $\phi_{1}$, the error estimate \eqref{eq:adj_priori_estimate_L^1} and the observation following its proof yield
\begin{equation}\label{eq:est_3h}
|\mathfrak{M}_{h}| \leq \|y_{1,h}\|_{L^{\infty}(\Omega)} \|\phi_{1,h}-\phi_{1}\|_{L^{1}(\Omega)} \|g\|_{L^{\infty}(\Omega)} \lesssim h^{2}|\log h|^{2} \|g\|_{L^{\infty}(\Omega)}.
\end{equation}
Combining \eqref{eq:est_1h}, \eqref{eq:est_2h}, and \eqref{eq:est_3h}, we conclude the desired estimate \eqref{eq:aux_j-jh_estimate_1}.

Let $g \in L^2(\Omega)$, let $y_{2,h} = \mathcal{S}_{h}(u_{2})$, and let $\phi_{2,h}$ be defined as $\phi_{1,h}$ with $y_{1,h}$ and $u_{1}$ replaced by $y_{2,h}$ and $u_{2}$, respectively. The bound \eqref{eq:aux_j-jh_estimate_2} follows from \eqref{eq:Djh}, the splitting $y_{2,h}\phi_{2,h} - y_{1,h}\phi_{1,h} = y_{2,h}[\phi_{2,h} - \phi_{1,h}] + \phi_{1,h}[y_{2,h} - y_{1,h}]$, and the discrete bounds $\|y_{2,h} - y_{1,h} \|_{L^{\infty}(\Omega)} + \|\phi_{2,h}-\phi_{1,h}\|_{L^{2}(\Omega)} \lesssim \|u_2-u_1\|_{L^{2}(\Omega)}$, which can be derived as in \cite[Lemma 6.3, Step 2]{MR4438718}. For brevity, we omit the details.
\end{proof}

\begin{lemma}[error bound for an auxiliary variable]\label{lem:auxiliary_variable}
Assume that \ref{A1}--\ref{A3} hold, that $f,a(\cdot,0) \in L^{\eta}(\Omega)$ for some $\eta > d$, and that $\bar{u}$ is a locally optimal control. Assume, in addition, that $\mathtt{a}>0$ if $d = 2$, or that $f - a(\cdot,0) \geq 0$ a.e.~in $\Omega$ and $f \not\equiv a(\cdot,0)$ if $d = 3$. Then, there exists $h_\star > 0$ such that, for every $h < h_\star$, there is
\begin{equation}\label{eq:aux_var_properties}
u_{h}^{*} \in \mathbb{U}_{ad,h}:
\qquad
j'(\bar u)(\bar u - u_h^*) = 0,
\quad
\|\bar u - u_h^*\|_{L^2(\Omega)} \leq Ch,
\end{equation}
where $C>0$ is independent of $h$.
\end{lemma}
\begin{proof}
Since the assumptions of Theorem \ref{thm:C01_reg_baru} are fulfilled, we have $\bar{u} \in C^{0,1}(\bar{\Omega})$. If $\bar{u}$ is constant, the choice $u_{h}^{*} = \bar{u}$ satisfies \eqref{eq:aux_var_properties} for every $h$, so we may assume that the Lipschitz constant $L_{\bar{u}}$ of $\bar{u}$ on $\bar{\Omega}$ is positive. Set $h_{\star} := (\mathtt{b}-\mathtt{a})/(2L_{\bar{u}})$ and let $h < h_{\star}$. Basic computations show that $|\bar{u}(x_{1})-\bar{u}(x_{2})| < (\mathtt{b}-\mathtt{a})/2$ for every $T \in \mathscr{T}_{h}$ and for every $x_1, x_2 \in T$.  Consequently, $\bar{u}$ cannot attain both control bounds on the same element, which, in view of \eqref{eq:sign_con_j_prime}, shows that $\bar{\mathfrak{p}} = \alpha\bar{u} - \bar{y}\bar{p}$ has constant sign a.e.~on each $T \in \mathscr{T}_{h}$. We now define $u_{h}^{*}$ as in \cite[(57)]{MR3586845}, with $\bar{d}$ replaced by $\bar{\mathfrak{p}}$; note that $\bar{\mathfrak{p}} \in L^{2}(\Omega)$. The constant sign of $\bar{\mathfrak{p}}$ allows us to apply a weighted mean value theorem, and the arguments in the proofs of \cite[Lemma 43]{MR3586845} and \cite[Lemma 6.4]{MR4438718} yield \eqref{eq:aux_var_properties}.
\end{proof}

\begin{theorem}[error estimate]\label{thm:error_estimate_control}
Assume that \ref{A1}--\ref{A3} hold and that $f,a(\cdot,0) \in L^{\infty}(\Omega)$. Assume, in addition, that $\mathtt{a}>0$ if $d = 2$, or that $f - a(\cdot,0) \geq 0$ a.e.~in $\Omega$ and $f \not\equiv a(\cdot,0)$ if $d = 3$. If $\bar{u} \in \mathbb{U}_{ad}$ is a locally optimal control satisfying the second order optimality condition \eqref{eq:equiv_cond}, then there exists $h_{\ddagger}>0$ such that
\begin{equation}\label{eq:error_estimate_control_2}
\|\bar{u}-\bar{u}_{h}\|_{L^{2}(\Omega)} \lesssim h | \log h| \quad \forall h<h_{\ddagger}, 
\end{equation}
where the hidden constant is independent of $h$.
\end{theorem}
\begin{proof}
As stated at the beginning of Section \ref{sec:error_estimates_fully}, we argue by contradiction. Let $\{h_{k}\}_{k \in \mathbb{N}} \subset \mathbb{R}^{+}$ be the sequence satisfying \eqref{eq:contradict_assumption}. By Lemma \ref{lem:aux_result}, \eqref{eq:aux_control_error_bound} holds for every $k \geq k_{\dagger}$. Since $h_k\to0$, there exists $k_\star\in\mathbb N$ such that $h_k< \min \{ h_\star, h_\odot \}$ for every $k\geq k_\star$, where $h_\star$ and $h_\odot $ are as in Lemmas \ref{lem:auxiliary_variable} and \ref{lem:estimates_for_j'_j'h}, respectively. Hence, in what follows, we consider $k\geq\max\{k_\dagger,k_\star\}$. Adding and subtracting $j'_{h_k}(\bar{u}_{h_k})(\bar{u}_{h_k}-\bar{u})$ on the right-hand side of \eqref{eq:aux_control_error_bound}, we obtain
\begin{equation}
\label{eq:control_estimate}
\begin{aligned}
\|\bar{u}-\bar{u}_{h_{k}}\|^{2}_{L^{2}(\Omega)}
&
\lesssim [j'(\bar{u}_{h_k})-j'_{h_{k}}(\bar{u}_{h_{k}})](\bar{u}_{h_k} - \bar{u})
+
[j'_{h_{k}}(\bar{u}_{h_{k}})-j'(\bar{u})](\bar{u}_{h_{k}} - \bar{u})
\\
& \lesssim h_{k}^{2} |\log h_{k}|^{2} +
[j'_{h_{k}}(\bar{u}_{h_{k}})-j'(\bar{u})](\bar{u}_{h_{k}} - \bar{u}).
\end{aligned}
\end{equation}
In the last step, we used \eqref{eq:aux_j-jh_estimate_1} with $u_1 = \bar{u}_{h_k}$ and $g=\bar{u}_{h_k}-\bar{u}$, together with $\|\bar{u}_{h_k}-\bar{u}\|_{L^\infty(\Omega)} \leq \mathtt{b}-\mathtt{a}$. It remains to control $[j'_{h_{k}}(\bar{u}_{h_{k}})-j'(\bar{u})](\bar{u}_{h_{k}} - \bar{u})$. To this end, we set $u=\bar{u}_{h_{k}}$ in \eqref{eq:char_var_ineq} and $u_{h} = u^{*}_{h_{k}}$ in \eqref{eq:fully_var_ineq}, where $u^*_{h_k}$ is provided by Lemma \ref{lem:auxiliary_variable}. Adding these inequalities yields
$
0 \leq j'(\bar{u})(\bar{u}_{h_{k}}-\bar{u}) + j'_{h_k}(\bar{u}_{h_{k}})(u^{*}_{h_{k}}-\bar{u}_{h_{k}})
= j'(\bar{u})(\bar{u}_{h_{k}}-\bar{u})
 + j'_{h_{k}}(\bar{u}_{h_{k}})(u^{*}_{h_{k}}-\bar{u}) + j'_{h_{k}}(\bar{u}_{h_{k}})(\bar{u}-\bar{u}_{h_{k}}).
$
From the latter, we deduce that
$
[j'(\bar{u})-j'_{h_{k}}(\bar{u}_{h_{k}})](\bar{u}-\bar{u}_{h_{k}}) \leq j'_{h_{k}}(\bar{u}_{h_{k}})(u^{*}_{h_{k}}-\bar{u}).
$
Using this bound, $j'(\bar{u})(u^{*}_{h_{k}}-\bar{u}) = 0$, which follows from \eqref{eq:aux_var_properties}, and adding and subtracting $j'_{h_{k}}(\bar{u})(u^{*}_{h_{k}}-\bar{u})$, we obtain
\begin{multline}
[j'(\bar{u})-j'_{h_{k}}(\bar{u}_{h_{k}})](\bar{u}-\bar{u}_{h_{k}})
\leq
j'_{h_{k}}(\bar{u}_{h_{k}})(u^{*}_{h_{k}}- \bar{u})
- j'(\bar{u})(u^{*}_{h_{k}}-\bar{u})
\label{eq:j'_j'h_calculus} \\
= [j'_{h_{k}}(\bar{u}_{h_{k}})-j'_{h_{k}}(\bar{u})](u^{*}_{h_{k}}-\bar{u}) + [j'_{h_{k}}(\bar{u})-j'(\bar{u})](u^{*}_{h_{k}}-\bar{u}). 
\end{multline}
Using \eqref{eq:aux_j-jh_estimate_2} with $u_1=\bar{u}_{h_k}$, $u_2=\bar{u}$, and
$g=u^*_{h_k}-\bar{u}$, and \eqref{eq:aux_j-jh_estimate_1} with $u_1=\bar{u}$ and $g=u^*_{h_k}-\bar{u}$, together with $\|u^*_{h_k}-\bar{u}\|_{L^\infty(\Omega)} \leq\mathtt{b}-\mathtt{a}$, we obtain
\begin{equation*}
  [j'(\bar{u})-j'_{h_{k}}(\bar{u}_{h_{k}})](\bar{u}-\bar{u}_{h_{k}})  \leq C \left(\|\bar{u}_{h_{k}}-\bar{u}\|_{L^{2}(\Omega)}\|u^{*}_{h_{k}}-\bar{u}\|_{L^{2}(\Omega)} + h_{k}^{2}|\log h_{k}|^{2} \right),
\end{equation*}
where $C$ is independent of $k$. Substituting this bound into \eqref{eq:control_estimate}, using Young's inequality and the estimate in \eqref{eq:aux_var_properties}, we deduce, for every $k\geq\max\{k_\dagger,k_\star\}$, that
\[
\|\bar{u}-\bar{u}_{h_{k}}\|^{2}_{L^{2}(\Omega)} \leq \tfrac{1}{2}\|\bar{u}-\bar{u}_{h_{k}}\|^{2}_{L^{2}(\Omega)} + 
\tilde{C} h_{k}^{2}(1+|\log h_{k}|^{2}),
\]
where $\tilde C>0$ is independent of $k$. Consequently, for all sufficiently large $k$, $\|\bar{u}-\bar{u}_{h_k}\|_{L^2(\Omega)} \lesssim h_k(1+|\log h_k|) \lesssim h_k|\log h_k|$, which contradicts \eqref{eq:contradict_assumption}.
\end{proof}

We improve the error estimate of Theorem \ref{thm:error_estimate_control} in two dimensions.

\begin{theorem}[improved error estimate]
\label{thm:improved_error_estimate}
Let $d=2$ and assume that the hypotheses of Theorem \ref{thm:error_estimate_control} hold. Then there exists $\tilde{h}_{\ddagger}>0$ such that
\begin{equation}\label{eq:error_estimate_h}
\|\bar{u}-\bar{u}_{h}\|_{L^{2}(\Omega)} \lesssim h \quad \forall h <\tilde{h}_{\ddagger},
\end{equation}
where the hidden constant is independent of $h$.
\end{theorem}
\begin{proof}
We proceed by contradiction. Recall that $\{\bar{u}_{h}\}_{h>0}$ converges to $\bar{u}$ in $L^{2}(\Omega)$ as $h \rightarrow 0$. Assume that \eqref{eq:error_estimate_h} does not hold. Then, for every $k \in \mathbb{N}$, there exists $h_k \in (0,k^{-1})$ such that $\|\bar{u}-\bar{u}_{h_{k}}\|_{L^{2}(\Omega)} >kh_{k}$. Consequently, $h_k \to 0$ as $k \to \infty$ and
\begin{equation}\label{eq:contradict_assumption_O(h)}
\lim_{k \to \infty}\|\bar{u}-\bar{u}_{h_{k}}\|_{L^{2}(\Omega)} = 0,
\qquad
\lim_{k \to \infty} h^{-1}_{k} \|\bar{u}-\bar{u}_{h_{k}}\|_{L^{2}(\Omega)} = + \infty.
\end{equation}
The estimate \eqref{eq:aux_control_error_bound} then follows from the arguments used in the proof of Lemma \ref{lem:aux_result}, with \eqref{eq:contradict_assumption_O(h)} in place of \eqref{eq:contradict_assumption}. We next establish a variant of \eqref{eq:aux_j-jh_estimate_1} that will be crucial for the improved error estimate. More precisely, if $u_1\in\mathbb{U}_{ad}$ and $g\in L^2(\Omega)$, then
\begin{equation}\label{eq:j'_j'h_improved_estimate}
|j'(u_{1})g-j_{h}'(u_{1})g| \lesssim h \|g\|_{L^{2}(\Omega)},
\end{equation}
provided $h$ is sufficiently small; the hidden constant
is independent of $u_1\in\mathbb U_{ad}$. The bound \eqref{eq:aux_j-jh_estimate_1} was established in Lemma \ref{lem:estimates_for_j'_j'h} upon decomposing $j'(u_{1})g - j_{h}'(u_{1})g = \mathfrak{M}_{h} + \mathfrak{D}_{h} + \mathfrak{E}_{h}$; we follow the same decomposition here. The estimates \eqref{eq:est_1h} and \eqref{eq:est_2h}, whose proofs only use $g \in L^{2}(\Omega)$, yield
\begin{equation*}
 |\mathfrak{E}_{h}|
\lesssim
 h^{2}|\log h|^{2}\|g\|_{L^{2}(\Omega)},
 \qquad
 |\mathfrak{D}_{h}|
\lesssim h^{2}|\log h|^{2} \|g\|_{L^{2}(\Omega)}.
\end{equation*}
It remains to bound $\mathfrak{M}_{h}$, the only term for which
\eqref{eq:aux_j-jh_estimate_1} required the norm  of $g$ in $L^{\infty}(\Omega)$. Using the uniform bound $\|y_{1,h}\|_{L^\infty(\Omega)}\lesssim1$ and \eqref{eq:adj_priori_estimate_L^2} for $d=2$, we obtain $|\mathfrak{M}_{h}| \leq \|y_{1,h}\|_{L^{\infty}(\Omega)} \|\phi_{1,h}-\phi_{1}\|_{L^{2}(\Omega)} \|g\|_{L^{2}(\Omega)} \lesssim h \|g\|_{L^{2}(\Omega)}$. Collecting these estimates and using that $h^{2}|\log h|^{2} \lesssim h$ for $h$
sufficiently small, we obtain \eqref{eq:j'_j'h_improved_estimate}.

To establish \eqref{eq:error_estimate_h}, we proceed as in the proof of Theorem \ref{thm:error_estimate_control}. In particular, using \eqref{eq:j'_j'h_improved_estimate} in the first inequality of \eqref{eq:control_estimate}, we obtain
\[
\|\bar{u}-\bar{u}_{h_{k}}\|^{2}_{L^{2}(\Omega)} \lesssim h_{k}\|\bar{u}-\bar{u}_{h_{k}}\|_{L^{2}(\Omega)} + [j'_{h_{k}}(\bar{u}_{h_{k}}) - j'(\bar{u})](\bar{u}_{h_{k}} - \bar{u}). 
\]
Estimates \eqref{eq:j'_j'h_calculus}, \eqref{eq:aux_j-jh_estimate_2}, \eqref{eq:j'_j'h_improved_estimate}, and \eqref{eq:aux_var_properties} yield
$
\|\bar{u}-\bar{u}_{h_{k}}\|^{2}_{L^{2}(\Omega)} \lesssim h_{k} \|\bar{u}-\bar{u}_{h_{k}}\|_{L^{2}(\Omega)} + h_k^2,
$
so Young's inequality gives $\|\bar{u}-\bar{u}_{h_{k}}\|_{L^2(\Omega)} \lesssim h_{k}$, which contradicts \eqref{eq:contradict_assumption_O(h)}.
\end{proof}

\subsection{Error estimates for the semidiscrete scheme}\label{sec:error_est_SD}
Let $\{\bar{\mathfrak{u}}_{h}\}_{h>0} \subset \mathbb{U}_{ad}$ be a sequence of locally optimal controls of the semidiscrete scheme such that $\bar{\mathfrak{u}}_{h} \rightarrow \bar{u}$ in $L^{2}(\Omega)$ as $h \rightarrow 0$, where $\bar{u} \in \mathbb{U}_{ad}$ is a locally optimal control of \eqref{eq:weak_ocp}--\eqref{eq:weak_state_eq}; see Section \ref{sec:semidiscete_convergence}. In the following, we establish the error estimate
\begin{equation}
 \|\bar{u}-\bar{\mathfrak{u}}_{h}\|_{L^{2}(\Omega)} \lesssim h |\log h| \quad \forall h\in (0,h_{\boxtimes}), \quad h_{\boxtimes} >0.
\end{equation}
As in the analysis of the fully discrete scheme, we begin with an auxiliary bound.

\begin{lemma}[auxiliary bound]\label{lem:auxiliary_bound_semidiscrete}
Assume that \ref{A1}--\ref{A3} hold and that $\bar{u}$ satisfies the second order optimality condition \eqref{eq:equiv_cond}. Then, there exists $h_{\dagger}>0$ such that
\begin{equation}\label{eq:aux_control_error_bound_semidiscrete}
\mathfrak{C} \|\bar{u}-\bar{\mathfrak{u}}_{h}\|^{2}_{L^{2}(\Omega)} \leq [j'(\bar{\mathfrak{u}}_{h})-j'(\bar{u})](\bar{\mathfrak{u}}_{h}-\bar{u}) \qquad \forall h < h_{\dagger}, \quad \mathfrak{C}=2^{-1}\min\{\mu,\alpha\}.
\end{equation} 
\end{lemma}
\begin{proof}
We proceed as in the proof of Lemma \ref{lem:aux_result}; the argument is simpler because
$\bar{u} \in \mathbb{U}_{ad}$ is admissible in \eqref{eq:first_opt_cond_semidiscrete} (no projection is needed). Assuming $\bar{\mathfrak{u}}_{h} \neq \bar{u}$ (otherwise \eqref{eq:aux_control_error_bound_semidiscrete} is trivial), define $g_{h}:=(\bar{\mathfrak{u}}_{h}-\bar{u})/\|\bar{\mathfrak{u}}_{h}-\bar{u}\|_{L^{2}(\Omega)}$ and extract a nonrelabeled subsequence such that $g_{h} \rightharpoonup g$ in $L^{2}(\Omega)$ as $h \rightarrow 0$. The arguments used in the proof of Lemma \ref{lem:aux_result} show that $g$ satisfies \eqref{eq:sign_cond}. We now show that $g(x)=0$ for a.e.~$x \in \Omega$ such that $\bar{\mathfrak{p}}(x) \neq 0$. Define $\bar{\mathtt{p}}_h := \alpha \bar{\mathfrak{u}}_h - \bar{y}_h \bar{p}_h$. By the same arguments as in the proof of Lemma \ref{lem:aux_result}, we have $\bar{\mathtt{p}}_h \to \bar{\mathfrak{p}}$ in $L^2(\Omega)$ as $h \to 0$. Setting
$u = \bar{u}$ in \eqref{eq:first_opt_cond_semidiscrete} yields $\int_{\Omega} \bar{\mathtt{p}}_h (\bar{\mathfrak{u}}_{h}-\bar{u})\,\mathrm{d}x \leq 0$ for every $h$. This, the aforementioned convergence, and $g_{h} \rightharpoonup g$ in $L^{2}(\Omega)$ as $h \to 0$ yield
\begin{equation}\label{eq:ineq_semi_aux}
\int_{\Omega} \bar{\mathfrak{p}}(x)g(x) \mathrm{d}x = \lim_{h \rightarrow 0}\frac{1}{\|\bar{\mathfrak{u}}_{h}-\bar{u}\|_{L^{2}(\Omega)}} \int_{\Omega}\bar{\mathtt{p}}_{h}(\bar{\mathfrak{u}}_{h}-\bar{u})\mathrm{d}x \leq 0.
\end{equation}
Arguing as in the proof of Lemma \ref{lem:aux_result}, we deduce that $g(x)=0$ for a.e.~$x \in \Omega$ such that $\bar{\mathfrak{p}}(x) \neq 0$. As a result, $g \in C_{\bar{u}}$. \emph{Step 2} in the proof of Lemma \ref{lem:aux_result} then applies verbatim. This yields
\eqref{eq:aux_control_error_bound_semidiscrete}.
\end{proof}

We establish an error bound for the semidiscrete scheme for $d \in \{2,3\}$.

\begin{theorem}[error estimate]\label{thm:error_semi}
Assume that \ref{A1}--\ref{A3} hold and that $f,a(\cdot,0) \in L^{\infty}(\Omega)$. If $\bar{u} \in \mathbb{U}_{ad}$ satisfies \eqref{eq:equiv_cond}, then there exists $h_{\ddagger}>0$ such that
\begin{equation}\label{eq:semi_error_hlogh}
\|\bar{u}-\bar{\mathfrak{u}}_{h}\|_{L^{2}(\Omega)} \lesssim h |\log h| \quad \forall h \in (0,h_{\ddagger}),
\end{equation} 
where the hidden constant is independent of $h$.
\end{theorem}
\begin{proof}
We set $u=\bar{\mathfrak{u}}_{h}$ in \eqref{eq:char_var_ineq} and $u = \bar{u}$ in \eqref{eq:first_opt_cond_semidiscrete} to obtain $-j'(\bar{u})(\bar{\mathfrak{u}}_{h}-\bar{u}) \leq 0$ and $-j'_{h}(\bar{\mathfrak{u}}_{h})(\bar{\mathfrak{u}}_{h}-\bar{u}) \geq 0$, respectively. Combining these relations with \eqref{eq:aux_control_error_bound_semidiscrete} yields
\begin{equation}\label{eq:thm87_estimate}
\mathfrak{C} \|\bar{u}-\bar{\mathfrak{u}}_{h}\|^{2}_{L^{2}(\Omega)} \leq [j'(\bar{\mathfrak{u}}_{h})-j'(\bar{u})](\bar{\mathfrak{u}}_{h}-\bar{u}) \leq  [j'(\bar{\mathfrak{u}}_{h})-j'_{h}(\bar{\mathfrak{u}}_{h})](\bar{\mathfrak{u}}_{h}-\bar{u}) \quad \forall h<h_{\dagger}.
\end{equation} 
Finally, since $\bar{\mathfrak{u}}_{h},\bar{u} \in \mathbb{U}_{ad} \subset L^{\infty}(\Omega)$, \eqref{eq:aux_j-jh_estimate_1} yields \eqref{eq:semi_error_hlogh} with $h_{\ddagger}:=\min\{h_{\dagger},h_{\odot}\}$.
\end{proof}

\begin{remark}[comparison with the fully discrete scheme]\label{rem:comparison_schemes}
We have proved the rate $\mathcal{O}(h|\log h|)$ for both schemes, but Theorem \ref{thm:error_semi} requires neither $\mathtt{a}>0$ when $d=2$ nor the sign conditions on $f-a(\cdot,0)$ when $d=3$: since $\bar{u}$ is admissible in \eqref{eq:first_opt_cond_semidiscrete}, neither Lemma \ref{lem:auxiliary_variable} nor Theorem \ref{thm:C01_reg_baru} is needed.
\end{remark}

In two dimensions, we now collect several properties of the discrete adjoint state.

\begin{lemma}[behavior of the discrete adjoint state near $\mathcal{E}$]\label{lem:discr_adj_behavior}
Let $d=2$ and assume that the hypotheses of Theorem \ref{thm:error_semi} hold. Let $\rho>0$ be such that $B_t(\rho) \Subset\Omega$ and $\bar{B}_{t}(\rho) \cap \mathcal{D} = \{t\}$ for every $t \in \mathcal{D}$. Given $M>0$, there exist $h_M>0$ and $\varrho_M \in (0,\rho]$ such that $|\bar{p}_{h}(x)| > M$ for every $t \in \mathcal{E}$, every $x \in B_{t}(\varrho_{M})$, and every $h < h_{M}$. Moreover, the following hold for $h$ sufficiently small:
\begin{itemize}[leftmargin=*,nosep]
\item[(i)] There exists a constant $C>0$ such that 
\begin{equation}\label{eq:log_bound_disc_adj}
\|\bar{p}_{h}\|_{L^{\infty}(\Omega)} \leq C |\log h|.
\end{equation} 
\item[(ii)] There exist $C_{1},C_{2}>0$ such that, for every $t \in \mathcal{E}$ and every $x \in B_{t}(\rho)\setminus\{t\}$,
\begin{equation}\label{eq:estimates_adjoint_cont-disc}
|\mathsf{p}(\bar{\mathfrak{u}}_{h})(x)-\bar{p}_{h}(x)|  \leq
\begin{cases}
\dfrac{C_{2}h^{2}}{|x-t|^{2}} \left|\log \left( \dfrac{|x-t|}{h}\right)\right|,  & |x-t| > C_{1} h,
\\
C_{2} \left( \left| \log |x-t|  \right| + 1 \right), &  0< |x-t| \leq C_{1} h,
\end{cases}
\end{equation}
\end{itemize}
where $\mathsf{p}(\bar{\mathfrak{u}}_{h})$ solves \eqref{eq:rhouh}. The constants $C$, $C_1$, and $C_2$ are independent of $h$.
\end{lemma}
\begin{proof}
Let $\bar{y}_{h} \in \mathbb{V}_{h}$ be the solution to \eqref{eq:state_eq_semidisc} with $\mathfrak{u} = \bar{\mathfrak{u}}_{h}$, and let $\bar{p}_{h} \in \mathbb{V}_{h}$ be the solution to \eqref{eq:discr_adj_eq} with $u_{h} = \bar{\mathfrak{u}}_{h}$ and $y_{h} = \bar{y}_{h}$. Since \eqref{eq:discr_adj_eq} is linear with respect to the adjoint variable and its right-hand side is a linear combination of Dirac measures,
\begin{equation}\label{eq:descom_ph_disc_green}
\bar{p}_{h} = \sum_{s \in \mathcal{D}}(\bar{y}_{h}(s) - y_{s})\bar{p}_{h,s},
\end{equation}
where, for each $s \in \mathcal{D}$, $\bar{p}_{h,s} \in \mathbb{V}_{h}$ solves
\begin{equation}
\label{eq:phs}
\int_{\Omega}\nabla w_{h} \cdot \nabla \bar{p}_{h,s} \mathrm{d}x
+
\int_{\Omega}\left[\frac{\partial a}{\partial y}(\cdot,\bar{y}_{h}) + \bar{\mathfrak{u}}_{h} \right]\bar{p}_{h,s} w_{h} \mathrm{d}x
=
\langle \delta_{s},w_{h} \rangle \quad \forall w_{h} \in \mathbb{V}_{h}.
\end{equation}

We now proceed in several steps.

\emph{Step 1}. Since $B_{t}(\rho) \Subset \Omega$ for every $t \in \mathcal{D}$, an adaptation of \cite[Lemma 5.7]{MR3973329} shows that, given $\mathfrak{M}>0$, there exist $h_{\mathfrak{M}}>0$ and $\varrho_{\mathfrak{M}} \in (0,\rho]$ such that
\begin{equation}\label{eq:disc_green_lbound}
\bar{p}_{h,t}(x) > \mathfrak{M} \qquad \forall t \in \mathcal{D}, \, \forall x \in B_{t}(\varrho_{\mathfrak{M}}), \, \forall h < h_{\mathfrak{M}}.
\end{equation}
In particular, $\bar{p}_{h,t}$ is positive on $B_{t}(\varrho_{\mathfrak{M}})$; this will
be used in \emph{Step 4}.

\emph{Step 2}. We show that the amplitudes associated with the points of $\mathcal{E}$ are bounded away from zero. Let $t \in \mathcal{E}$, and set $y(\bar{\mathfrak{u}}_{h})=\mathcal{S}(\bar{\mathfrak{u}}_{h})$. In light of \eqref{eq:I_h}, \eqref{eq:state_priori_estimate_L^infty}, and \eqref{eq:semi_error_hlogh},
\begin{multline} \label{eq:yht-yt_pos}
0 < |\bar{y}(t) - y_{t}| \leq |\bar{y}(t)-y(\bar{\mathfrak{u}}_{h})(t)|
+ |y(\bar{\mathfrak{u}}_{h})(t)-\bar{y}_{h}(t)| + |\bar{y}_{h}(t)-y_{t}|
\\
\lesssim \|\bar{u}-\bar{\mathfrak{u}}_{h}\|_{L^{2}(\Omega)} + h
+ |\bar{y}_{h}(t)-y_{t}| \leq C\left( h|\log h| + |\bar{y}_{h}(t)-y_{t}| \right),
\end{multline}
where $C>0$ is independent of $t$ and $h$, and $h$ is sufficiently small. Since $\mathcal{E}$ is finite, there exist $\bar{C}>0$, independent of $t$ and $h$, and $h_{\wedge}>0$, independent of $t$, such that
\begin{equation}
\label{eq:yht-yt}
 |\bar{y}_{h}(t)-y_{t}| \geq \bar{C} |\bar{y}(t)-y_{t}| > 0
 \qquad
 \forall t \in \mathcal{E}, \, \forall h < h_{\wedge}.
\end{equation}
On the other hand, since $\|\bar{y}_{h}\|_{L^{\infty}(\Omega)} \lesssim 1$ uniformly in $h$, $|\bar{y}_{h}(s)-y_{s}| \lesssim 1$ for all $s \in \mathcal{D}$.

\emph{Step 3}. We now derive pointwise bounds for $\bar{p}_{h,s}$. Since $\mathcal{D}$ is finite and $\bar{B}_t(\rho)\cap\mathcal{D} =\{t\}$ for every $t\in\mathcal{D}$, there exists $\gamma>0$ such that
\begin{equation}
\label{eq:xminus}
 |x - s| \geq \gamma \qquad \forall t \in \mathcal{D}, \ \forall x \in B_t(\rho), \ \forall s \in \mathcal{D} \setminus \{t \}.
\end{equation}
For $s \in \mathcal{D}$, let $p_{s} \in W^{1,r}_{0}(\Omega)$ be the continuous counterpart of $\bar{p}_{h,s}$. In view of \eqref{eq:xminus} and the bound $|p_{s}(x)| \lesssim |\log|x-s||+1$, which holds uniformly in $h$ because the reaction coefficient is nonnegative and uniformly bounded in $L^{\infty}(\Omega)$, we have $\|p_{s}\|_{L^{\infty}(B_{t}(\rho))} \lesssim 1$, uniformly in $s$, $t$ and $h$. This bound and the interior maximum-norm error estimate away from $s$ in \cite[Theorem 6.1]{MR431753} imply that there exists $\mathfrak{N}>0$, independent of $h$, $s$, and $t$, such that
\begin{equation}
\label{eq:phsLinfinitylocal}
|\bar{p}_{h,s}(x)| \leq |\bar{p}_{h,s}(x) - p_s(x)| + |p_s(x)| \leq \mathfrak{N}  \qquad \forall t \in \mathcal{D}, \ \forall x \in B_t(\rho), \ \forall s \in \mathcal{D} \setminus \{t \},
\end{equation}
for $h$ sufficiently small.

\emph{Step 4}. Let $t \in \mathcal{E}$. Set $\mathfrak{I}_s= \bar{y}_h(s) - y_{s}$ for $s \in \mathcal{D}$. In view of
\eqref{eq:descom_ph_disc_green}, we write
\[
 \bar{p}_h(x) = \mathfrak{I}_t \bar{p}_{h,t}(x) + \sum_{s \in \mathcal{E}\setminus \{ t\}} \mathfrak{I}_s \bar{p}_{h,s}(x) + \sum_{s \in \mathcal{D}\setminus \mathcal{E}}\mathfrak{I}_s \bar{p}_{h,s}(x) =: \mathfrak{J} + \mathfrak{K} + \mathfrak{L},
\]
so that $|\bar{p}_h(x)| \geq |\mathfrak{J}| - |\mathfrak{K}| - |\mathfrak{L}|$. A bound for $|\mathfrak{J}|$ follows from \eqref{eq:yht-yt} and \eqref{eq:disc_green_lbound}: $|\mathfrak{J}| = |\mathfrak{I}_t| \bar{p}_{h,t}(x) \geq \bar{C} |\bar{y}(t)-y_{t}|\mathfrak{M}$ for all $x \in B_t( \varrho_{\mathfrak{M}} )$ and for all $h < \min\{ h_{\mathfrak{M}}, h_{\wedge} \}$. To bound $\mathfrak{K}$, we note that $|x-s| \geq \gamma$ for every $s \in \mathcal{E}\setminus\{t\}$ and every $x \in B_{t}(\varrho_{\mathfrak{M}})$, with $\gamma$ as in \emph{Step 3}. An application of \eqref{eq:phsLinfinitylocal} and the bound $|\mathfrak{I}_{s}| \lesssim 1$ of \emph{Step 2} thus shows
\[
 |\mathfrak{K}|
 \leq \sum_{s \in \mathcal{E}\setminus\{t\}} |\mathfrak{I}_{s}| \| \bar{p}_{h,s}\|_{L^{\infty}(B_t(\varrho_{\mathfrak{M}}))} \leq \hat{C} \mathfrak{N},
\]
where $\hat{C}$ is independent of $h$, $t$, and of $\mathfrak{M}$. We finally control $\mathfrak{L}$. Since $\bar{y}(s)=y_{s}$ for every $s \in \mathcal{D} \setminus \mathcal{E}$, \eqref{eq:I_h} and \eqref{eq:semi_error_hlogh} give $|\bar{y}_h(s) - \bar{y}(s)| \lesssim h |\log h|$. This and the local bound \eqref{eq:phsLinfinitylocal} yield
\[
 |\mathfrak{L}|
 \leq \sum_{s \in \mathcal{D}\setminus\mathcal{E} } |\bar{y}_h(s) - \bar{y}(s)| \| \bar{p}_{h,s}\|_{L^{\infty}(B_{t}(\varrho_{\mathfrak{M}}))} \lesssim h |\log h|,
\]
where the hidden constant is independent of $h$, $t$, and $\mathfrak M$.
Collecting the three bounds yields
$
 |\bar{p}_{h}(x)| \geq \bar{C}\sigma\,\mathfrak{M} - \hat{C}\mathfrak{N} - C h|\log h|,
$
where $\sigma:= \min_{s\in\mathcal E}|\bar y(s)-y_s|>0$. Since $\hat{C}\mathfrak{N}$ and $C h|\log h|$ are independent of $\mathfrak{M}$ and $C h|\log h| \to 0$, choosing first $\mathfrak{M}$ sufficiently large and then $h_M>0$ sufficiently small yields $|\bar{p}_{h}(x)| > M$ for every $t \in \mathcal{E}$, every $x \in B_{t}(\varrho_{M})$, and every $h < h_{M}$, with $\varrho_{M} := \varrho_{\mathfrak{M}}$ and $h_{M} \leq \min\{h_{\mathfrak{M}}, h_{\wedge}\}$.

\emph{Step 5.} Adapting \cite[Exercise 8.x.19]{MR2373954} yields $\|\bar{p}_{h,s}\|_{L^{\infty}(\Omega)} \lesssim |\log h|$ for every
$s \in \mathcal{D}$, with a hidden constant independent of $h$ and $s$. This, \eqref{eq:descom_ph_disc_green}, and the bound $|\mathfrak{I}_{s}| \lesssim 1$ of \emph{Step 2} yield \eqref{eq:log_bound_disc_adj}. Finally,
\eqref{eq:estimates_adjoint_cont-disc} follows from \cite[Theorem 6.1]{MR431753}.
\end{proof}

\begin{lemma}[coincidence of $\bar{\mathfrak{u}}_{h}$ and $\bar{u}$ near $\mathcal{E}$]\label{lem:equality_u_uh}
Let $d=2$ and $\mathtt{a} >0$ and suppose that the hypotheses of Theorem \ref{thm:error_semi} hold. Then, there exist $h_{\ddagger \ddagger}>0$ and $\varsigma \in (0,\rho]$, with $\rho$ as in Lemma \ref{lem:discr_adj_behavior}, such that $\bar{\mathfrak{u}}_{h}(x) = \bar{u}(x)$ for every $t \in \mathcal{E}$, every $x \in B_t(\varsigma)$, and every $h < h_{\ddagger \ddagger}$.
\end{lemma}

\begin{proof}
For the sake of simplicity, we assume that $\mathcal{D} = \mathcal{E} = \{t\}$; the general case follows by arguments similar to those in \emph{Step 4} of the proof of Lemma \ref{lem:discr_adj_behavior}. We divide the analysis into two cases: $\bar{y}(t) \neq 0$ and $\bar{y}(t) = 0$, where $\bar{y} = \mathcal{S}(\bar{u})$.

\emph{Case 1. $\bar{y}(t) \neq 0$.} Since $\bar{\mathfrak{u}}_{h} \to \bar{u}$ in $L^{2}(\Omega)$ as $h \to 0$ (see the beginning of Section \ref{sec:error_est_SD}), Theorem \ref{thm:aux_error_estimate_state}, which extends verbatim to sequences in $\mathbb{U}_{ad}$, implies that $\bar{y}_{h} \rightarrow \bar{y}$ in $C(\bar{\Omega})$ as $h \rightarrow 0$, where $\bar{y}_{h}=\mathcal{S}_{h}(\bar{\mathfrak{u}}_{h})$.
Since $\bar{y} \in C(\bar \Omega)$ (cf.~Theorem \ref{thm:state_reg}) and $\bar{y}(t) \neq 0$, there exists $\vartheta>0$ such that $|\bar{y}(x)| \geq |\bar{y}(t)|/2$ for every $x \in B_{t}(\vartheta)$. This and $\|\bar{y}-\bar{y}_{h}\|_{L^{\infty}(\Omega)} \leq Ch|\log h|$, which follows from Theorems \ref{thm:aux_error_estimate_state} and \ref{thm:error_semi}, yield
\begin{equation*}
|\bar{y}_{h}(x)| \geq |\bar{y}(x)| - |\bar{y}_{h}(x) - \bar{y}(x)| \geq |\bar{y}(t)|/2 - Ch|\log h| \geq |\bar{y}(t)|/4 >0,
\end{equation*}
for every $x \in B_{t}(\vartheta)$ and $h$ sufficiently small. Define $\zeta := |\bar{y}(t)|/4 > 0$ and choose $M>\alpha\zeta^{-1} \mathtt{b}$. An application of Lemma \ref{lem:discr_adj_behavior} yields $h_{M}>0$ and $\varrho_{M} \in (0,\rho]$ such that
\begin{equation}\label{eq:discr_adj_ex}
|\bar{p}_{h}(x)| > M \qquad \forall h < h_{M}, \, \forall x \in B_{t}(\varrho_M).
\end{equation}
Here, $\bar{p}_{h} \in \mathbb{V}_{h}$ denotes the solution of \eqref{eq:discr_adj_eq} with $u_{h} = \bar{\mathfrak{u}}_{h}$ and $y_{h} = \bar{y}_{h}$. Set $ \varepsilon :=\min\{\vartheta,\varrho_M\}/2$. After reducing $h_{M}$ if necessary, so that $|\bar{y}_{h}| \geq \zeta$ in $B_{t}(\varepsilon)$, we obtain
$
\alpha^{-1} |\bar{y}_{h}(x) \bar{p}_{h}(x)| > \alpha^{-1} \zeta M  > \alpha^{-1} \zeta (\alpha\zeta^{-1}\mathtt{b}) = \mathtt{b}
$
for every $x \in B_{t}(\varepsilon)$ and every $h < h_{M}$.
Since $\bar{y}_{h}\bar{p}_{h} \in C(\bar{\Omega})$ and does not vanish in $B_t(\varepsilon)$, it has constant sign there. Recall that $0 < \mathtt{a} < \mathtt{b}$. Therefore, the projection formula \eqref{eq:disc_opt_ct_SD} yields
\[
\bar{\mathfrak{u}}_{h}\equiv\mathtt{a}
\quad \text{or}\quad
\bar{\mathfrak{u}}_{h}\equiv\mathtt{b}
\quad \text{in }B_t(\varepsilon).
\]

Since $t \in \mathcal{E}$, $\bar{y}(t) - y_t \neq 0$. The convergence $\bar{y}_{h}\rightarrow\bar{y}$ in $C(\bar{\Omega})$ as
$h\rightarrow0$ then shows that $\operatorname{sign}(\bar{y}_{h}(t)-y_{t}) = \operatorname{sign}(\bar{y}(t)-y_{t})$ for $h$ sufficiently small. On the other hand, since $\mathcal{D} = \mathcal{E} = \{ t\}$, \eqref{eq:descom_ph_disc_green} reads $\bar{p}_{h} = (\bar{y}_{h}(t)-y_{t})\bar{p}_{h,t}$ in $B_{t}(\varepsilon)$, where $\bar{p}_{h,t}$ solves \eqref{eq:phs} with $s = t$. Using that $\bar{p}_{h,t}>0$ in $B_{t}(\varrho_{M}) \supset B_{t}(\varepsilon)$ (cf.~\emph{Step 1} in the proof of Lemma \ref{lem:discr_adj_behavior}), we obtain $\operatorname{sign}(\bar{p}_{h}) = \operatorname{sign}(\bar{y}_{h}(t)-y_{t})$ in $B_t(\varepsilon)$, provided $h$ is sufficiently small. A similar argument applied to $\bar{p} = (\bar{y}(t) - y_t)\bar{p}_t$ and the positivity of the continuous Green function $\bar{p}_{t}$ gives $\operatorname{sign}(\bar{p}) = \operatorname{sign}(\bar{y}(t)-y_{t})$ in $B_{t}(\varepsilon)\setminus\{t\} $. Since $|\bar{y}| \geq |\bar{y}(t)|/2$ in $B_{t}(\varepsilon)$ and $\bar{y}_{h} \to \bar{y}$ in $C(\bar{\Omega})$, we also have $\operatorname{sign}(\bar{y}_{h}) = \operatorname{sign}(\bar{y})$ in $B_t(\varepsilon)$ for $h$ sufficiently small. Combining these relations, $\operatorname{sign}(\bar{y}_{h}\bar{p}_{h}) = \operatorname{sign}(\bar{y}\bar{p})$ in $B_{t}(\varepsilon) \setminus\{t\}$. Moreover, after reducing $\varepsilon$ if necessary, $\alpha^{-1}|\bar{y}\bar{p}| > \mathtt{b}$ in $B_{t}(\varepsilon) \setminus \{t\}$. This follows from $|\bar{y}| \geq |\bar{y}(t)|/2$ and $|\bar{p}(x)| \to \infty$ as $x \to t$ (cf.~\cite[Theorem 5.2, Case 1]{MR5008466}). Finally, combining
$\operatorname{sign}(\bar{y}_{h}\bar{p}_{h}) = \operatorname{sign}(\bar{y}\bar{p})$ with $\alpha^{-1}|\bar{y}_{h}\bar{p}_{h}|>\mathtt{b}$ and $\alpha^{-1}|\bar{y}\bar{p}|>\mathtt{b}$, the continuous and discrete projection formulas yield $\bar{\mathfrak{u}}_{h} = \bar{u}$ in $B_{t}(\varepsilon)\setminus\{t\}$, and, by continuity, in all of $B_{t}(\varepsilon)$. We thus set $\varsigma := \varepsilon$.

\emph{Case 2. $\bar{y}(t) = 0$.} Recall that $B_{t}(\rho) \Subset \Omega$, with $\rho$ as in Lemma \ref{lem:discr_adj_behavior}. Given $x \in B_{t}(\rho)\setminus\{t\}$, we write
\begin{multline}\label{eq:main_bound_product}
|\bar{y}_{h}(x)\bar{p}_{h}(x)|  \leq |(\bar{y}_{h}(x)-\bar{y}(x))\bar{p}_{h}(x)| + |\bar{y}(x)(\bar{p}_{h}(x)-\mathsf{p}(\bar{\mathfrak{u}}_h)(x))|
\\
+ |\bar{y}(x)\mathsf{p}(\bar{\mathfrak{u}}_h)(x)|
 =: \mathfrak{O} + \mathfrak{P} + \mathfrak{R}. 	
\end{multline}
We begin by controlling $\mathfrak{O}$. For this purpose, we use the $L^{\infty}(\Omega)$-error estimate in \eqref{eq:fully_state_estimates}, the global bound \eqref{eq:log_bound_disc_adj}, and the error estimate \eqref{eq:semi_error_hlogh}:
\begin{equation}\label{eq:bound_O}
\mathfrak{O} \lesssim (h + \|\bar{u}-\bar{\mathfrak{u}}_{h}\|_{L^{2}(\Omega)}) |\bar{p}_{h}(x)|  \lesssim h |\log h|^{2},
\end{equation}
which holds provided $h$ is sufficiently small. We now control $\mathfrak{P}$. In light of Lemma \ref{lem:discr_adj_behavior}, we distinguish two cases: $0 < |x-t| \leq C_{1} h $ and $|x-t| > C_{1}h$.

\emph{Case 2.1:} $0 < |x-t| \leq C_{1} h $. Since $f,a(\cdot,0) \in L^{\infty}(\Omega)$, \ref{A1}--\ref{A3} hold, and $\Omega$ is convex, an application of \cite[Lemma 4.1]{MR3973329} shows that $\bar{y} \in C^{0,1}(\bar{\Omega})$. This, the assumption $\bar{y}(t) = 0$, and the second estimate in \eqref{eq:estimates_adjoint_cont-disc} yield
\begin{equation}\label{eq:bound_P_near}
\mathfrak{P} = |\bar{y}(x)-\bar{y}(t)| |\bar{p}_{h}(x) - \mathsf{p}(\bar{\mathfrak{u}}_h)(x)|
\lesssim
|x-t| \cdot (|\log |x-t|| + 1).
\end{equation}

\emph{Case 2.2:} $|x-t|>C_{1} h$. In this setting, the first estimate in \eqref{eq:estimates_adjoint_cont-disc} away from $t \in \mathcal{E}$ yields
\begin{equation*}
\mathfrak{P} \lesssim |x-t| \cdot |\bar{p}_{h}(x) - \mathsf{p}(\bar{\mathfrak{u}}_h)(x)|
\lesssim h \left[\frac{h}{|x-t|} \left|\log\left(\frac{|x-t|}{h}\right)\right| \right].
\end{equation*}
To bound the term in brackets, we set $\varpi=|x-t|/h$ and note that $\varpi > C_1$. The function $\varpi \mapsto \varpi^{-1}|\log \varpi|$ is bounded on $[C_1,\infty)$ by a constant depending only on $C_1$. Consequently,
\begin{equation}\label{eq:bound_P_far}
\mathfrak{P} \lesssim h  \lesssim h |\log h|^2.
\end{equation}

We finally bound $\mathfrak{R}$ using the asymptotic behavior of $\mathsf{p}(\bar{\mathfrak{u}}_h)$ near $t \in \mathcal{E}$ \cite{MR3169756}:
$
|\mathsf{p}(\bar{\mathfrak{u}}_h)(x)| \lesssim |\log |x-t|| +1.
$
Since the coefficients in \eqref{eq:rhouh} are uniformly bounded with respect to $h$ in $L^\infty(\Omega)$, the hidden constant in the previous estimate does not depend on $h$. The fact that $\bar{y} \in C^{0,1}(\bar \Omega)$ combined with $\bar{y}(t) = 0$ thus yields
\begin{equation}\label{eq:bound_R}
\mathfrak{R} \lesssim |x-t| \cdot (|\log |x-t|| + 1).
\end{equation}
Substituting the bounds \eqref{eq:bound_O}, \eqref{eq:bound_P_near}, \eqref{eq:bound_P_far}, and \eqref{eq:bound_R} into \eqref{eq:main_bound_product}, we obtain
\begin{equation}\label{eq:final_combined_bound}
|\bar{y}_{h}(x) \bar{p}_{h}(x)| \lesssim h |\log h|^2 + |x-t| \cdot (|\log |x-t|| + 1),
\end{equation}
for every $x\in B_t(\rho)\setminus\{t\}$, after reducing $\rho$ if necessary, where $\rho$ is independent of $h$.

The bound \eqref{eq:final_combined_bound} yields the existence of $h_{\ddagger \ddagger} >0$ and $\varepsilon \in (0,\rho]$ such that
$
\alpha^{-1} |\bar{y}_{h}(x) \bar{p}_{h}(x)| < \mathtt{a}
$
for all $x\in B_t(\varepsilon)\setminus\{t\}$, and for all $h < h_{\ddagger \ddagger}$.
Consequently, the projection formula \eqref{eq:disc_opt_ct_SD} yields
$\bar{\mathfrak u}_h=\mathtt a$ in $B_{t}(\varepsilon) \setminus \{ t\}$. Since $\bar{y}(t)=0$, the arguments developed in \cite[Theorem 5.2, Case 2]{MR5008466} reveal that $\bar{u} = \mathtt{a}$ in $B_{t}(\lambda)$, for some $\lambda >0$. Setting $\varsigma := \min \{\varepsilon,\lambda\}$,  we obtain that $\bar{\mathfrak{u}}_{h} = \bar{u}$ in $B_{t}(\varsigma) \setminus \{ t \}$, and, by continuity, in all of $B_{t}(\varsigma)$.
\end{proof}

In two dimensions, the rate of Theorem \ref{thm:error_semi} can be improved to $h^{2}|\log h|^{2}$.

\begin{theorem}[improved error estimate]\label{thm:ee_hsquare_semidisc}
Let $d=2$ and $\mathtt{a}>0$, and assume that the hypotheses of Theorem \ref{thm:error_semi} hold. Then there exists $h_{\ominus}>0$ such that
\begin{equation}\label{eq:ee_hsquare_semidisc}
\|\bar{u}-\bar{\mathfrak{u}}_{h}\|_{L^{2}(\Omega)} \lesssim h^{2} |\log h|^{2} \quad \forall h < h_{\ominus},
\end{equation}  
where the hidden constant is independent of $h$.
\end{theorem}
\begin{proof}
We split the analysis into the following two cases: $\mathcal{E} = \mathcal{D}$ and $\mathcal{E} \subsetneq \mathcal{D}$.

\emph{Case 1: $\mathcal{E} = \mathcal{D}$.} In view of \eqref{eq:thm87_estimate}, it
suffices to prove that
\begin{equation}\label{eq:thm810_estimate}
[j'(\bar{\mathfrak{u}}_{h})-j'_{h}(\bar{\mathfrak{u}}_{h})](\bar{\mathfrak{u}}_{h} - \bar{u}) \lesssim h^{2}|\log h|^{2}\|\bar{u}-\bar{\mathfrak{u}}_{h}\|_{L^{2}(\Omega)}.
\end{equation}
For this purpose, we use the decomposition in \eqref{eq:decomposition} of Lemma \ref{lem:estimates_for_j'_j'h} with $u_1 = \bar{\mathfrak{u}}_{h}$ and $g = \bar{\mathfrak{u}}_{h}-\bar{u}$: $j'(\bar{\mathfrak{u}}_{h})(\bar{\mathfrak{u}}_{h}-\bar{u})- j'_{h}(\bar{\mathfrak{u}}_{h})(\bar{\mathfrak{u}}_{h}-\bar{u})
= \mathfrak{M}_{h} + \mathfrak{D}_{h} + \mathfrak{E}_{h}$. The bounds \eqref{eq:est_1h} and \eqref{eq:est_2h} control $\mathfrak{E}_{h}$ and $\mathfrak{D}_{h}$; it thus remains to bound $\mathfrak{M}_{h}$. In our setting, $\mathfrak{M}_{h} = (\bar{y}_{h}[\bar{p}_{h}-\mathsf{p}(\bar{\mathfrak{u}}_{h})],\bar{\mathfrak{u}}_{h}-\bar{u})_{L^{2}(\Omega)}$, where $\bar{y}_{h} = \mathcal{S}_{h}(\bar{\mathfrak{u}}_{h})$, $\bar{p}_{h}$ solves \eqref{eq:discr_adj_eq}, and $\mathsf{p}(\bar{\mathfrak{u}}_{h})$ solves \eqref{eq:rhouh}. Note that $\bar{p}_h$ is precisely the Galerkin approximation of $\mathsf{p}(\bar{\mathfrak{u}}_{h})$. Since $\mathcal{E} = \mathcal{D}$, Lemma \ref{lem:equality_u_uh} guarantees that
$\bar{\mathfrak{u}}_{h} = \bar{u}$ in $\cup_{t \in \mathcal{D}}B_{t}(\varsigma)$ for $h$ sufficiently small. Consequently,
$
\mathfrak{M}_{h} = (\bar{y}_{h}[\bar{p}_{h}- \mathsf{p}(\bar{\mathfrak{u}}_{h})],\bar{\mathfrak{u}}_{h} - \bar{u})_{L^{2}(\Omega\setminus \cup_{t \in \mathcal{D}}\bar{B}_{t}(\varsigma))}.
$
We use \eqref{eq:error_est_adjoint_out_balls}, which holds uniformly for nonnegative reaction  coefficients bounded uniformly in $L^{\infty}(\Omega)$, and the uniform boundedness of $\bar{y}_{h}$ in $L^{\infty}(\Omega)$ to obtain
\begin{equation*}
|\mathfrak{M}_{h}| \lesssim \|\mathsf{p}(\bar{\mathfrak{u}}_{h}) - \bar{p}_{h}\|_{L^{2}(\Omega\setminus \cup_{t \in \mathcal{D}}\bar{B}_{t}(\varsigma))} \|\bar{u}-\bar{\mathfrak{u}}_{h}\|_{L^{2}(\Omega\setminus \cup_{t \in \mathcal{D}}\bar{B}_{t}(\varsigma))} \lesssim h^{2} |\log h|^{2} \|\bar{u}-\bar{\mathfrak{u}}_{h}\|_{L^{2}(\Omega)}.
\end{equation*}
Combining this bound with the estimates for $\mathfrak{E}_{h}$ and $\mathfrak{D}_{h}$ yields \eqref{eq:thm810_estimate}. The desired estimate \eqref{eq:ee_hsquare_semidisc} then follows from \eqref{eq:thm87_estimate}.

\emph{Case 2: $\mathcal{E} \subsetneq \mathcal{D}$.} As in \emph{Case 1}, Lemma
\ref{lem:equality_u_uh} gives $\bar{u} = \bar{\mathfrak{u}}_{h}$ in
$\cup_{t \in \mathcal{E}}B_{t}(\varsigma)$ for $h$ sufficiently small, whence $
| \mathfrak{M}_{h} | \lesssim \|\mathsf{p}(\bar{\mathfrak{u}}_{h}) - \bar{p}_{h}\|_{L^{2}(\Omega \setminus \cup_{t \in \mathcal{E}}\bar{B}_{t}(\varsigma))} \|\bar{u}-\bar{\mathfrak{u}}_{h}\|_{L^{2}(\Omega \setminus \cup_{t \in \mathcal{E}}\bar{B}_{t}(\varsigma))}$. The estimate \eqref{eq:error_est_adjoint_out_balls} is not directly applicable here, because it holds on $\Omega \setminus \cup_{t \in \mathcal{D}}\bar{B}_{t}(\varsigma) \subsetneq \Omega \setminus \cup_{t \in \mathcal{E}}\bar{B}_{t}(\varsigma)$. We thus proceed as follows:
\[
 \|\mathsf{p}(\bar{\mathfrak{u}}_{h}) - \bar{p}_{h}\|_{L^{2}(\Omega \setminus \cup_{t \in
 \mathcal{E}}\bar{B}_{t}(\varsigma))} \leq \|\mathsf{p}(\bar{\mathfrak{u}}_{h}) -
 \bar{p}_{h}\|_{L^{2}(\Omega \setminus \cup_{t \in \mathcal{D}}\bar{B}_{t}(\varsigma))}
 + \|\mathsf{p}(\bar{\mathfrak{u}}_{h}) - \bar{p}_{h}\|_{L^{2}(\cup_{t \in
 \mathcal{D} \setminus \mathcal{E}}\bar{B}_{t}(\varsigma))}.
\]
The first term is bounded by $h^{2}|\log h|^{2}$ in view of
\eqref{eq:error_est_adjoint_out_balls}. It remains to control the second one. To this end, set $\Sigma := \cup_{t \in \mathcal{D}\setminus\mathcal{E}}\bar{B}_{t}(\varsigma)$ and $e_{\mathsf{p}}:= \mathsf{p}(\bar{\mathfrak{u}}_h) - \bar{p}_h$, and adapt the arguments in \cite[Remark 6.11]{MR4438718}
and \cite[Lemma 5.5]{MR3973329}; the latter with the indicator function $\chi_{\Sigma}$ of $\Sigma$ in place of $1-\chi_{B}$. Thus,
\[
\|e_\mathsf{p}\|^{2}_{L^{2}(\Sigma)} \lesssim h^{2}|\log h|^{2}\|e_{\mathsf{p}}\|_{L^{2}(\Sigma)}
+ h\|e_\mathsf{p}\|_{L^{2}(\Sigma)}\|\bar{y} - \bar{y}_{h}\|_{L^{\infty}(\Omega)},
\]
where we have used \cite[Lemma 4.4(ii)]{MR3973329}. We now apply \eqref{eq:fully_state_estimates} and \eqref{eq:semi_error_hlogh} to deduce that $\|\bar{y} - \bar{y}_{h}\|_{L^{\infty}(\Omega)} \lesssim h |\log h|$ and thus that $\|e_{\mathsf{p}}\|_{L^{2}(\Sigma)} \lesssim h^{2}|\log h|^{2}$. Consequently, $\mathfrak{M}_h$ satisfies the same bound as in \emph{Case 1}. Combining this estimate with the bounds for $\mathfrak{E}_h$ and $\mathfrak{D}_h$ yields \eqref{eq:thm810_estimate}, and \eqref{eq:ee_hsquare_semidisc} follows from \eqref{eq:thm87_estimate}.
\end{proof}


\section{Numerical examples}\label{sec:num_examples}

We present two- and three-dimensional numerical experiments illustrating the performance of the fully discrete and semidiscrete schemes.

\subsection{Implementation details}

The numerical experiments were carried out using a code that we implemented in \texttt{C++}. All integrals involving $f$, $a(\cdot,y)$, and $\Pi_{[\mathtt{a},\mathtt{b}]}$, as well as the approximation errors, were computed using a quadrature rule exact for polynomials of degree $19$ when $d=2$ and degree $14$ when $d=3$. For each scheme, the resulting nonlinear system was solved using an adaptation of the semismooth Newton method in \cite[Appendix A.1]{MR2971171}, together with the multifrontal massively parallel sparse direct solver (MUMPS) \cite{MUMPS1,MUMPS2}.

The implementation of the semidiscrete scheme requires assembling terms where $\mathfrak{\bar{u}}_{h} = \Pi_{[\mathtt{a},\mathtt{b}]}(\alpha^{-1}\bar{y}_{h}\bar{p}_{h})$ exhibits kinks within the elements of $\mathscr{T}_{h}$; exact integration would require locating the (generally curved) active/inactive regions in such elements; see \cite[Remark 5.19]{MR2536007}. We instead compute such terms with our quadrature rule. This yields an \emph{approximate version} of the semidiscrete scheme, which nonetheless delivers, as reported below, optimal experimental rates of convergence; see Figure \ref{fig:ex-1.1}.

\subsection{General setting and exact solutions}

We set $\Omega=(0,1)^d$, with $d\in\{2,3\}$, $a(\cdot,y)=\sinh(y)$, $\alpha=1$, $\mathtt{a}=10^{-2}$, $\mathtt{b}=1$, and $y_t=-1$; note that $a(\cdot,0)=0$. We consider $\mathcal{D}=\{(\tfrac12,\tfrac12)\}$ when $d=2$ and $\mathcal{D}=\{(\tfrac12,\tfrac12,\tfrac12)\}$ when $d=3$.

We slightly modify the cost functional by setting $\mathfrak{J}(y,u):=J(y,u)+\int_\Omega gy\,\mathrm{d}x$, where $g \in L^{\infty}(\Omega)$. This modification only affects the adjoint equation, which becomes
\begin{equation}\label{weak_mod_adj_eq}
\int_\Omega\nabla w\cdot\nabla p\mathrm{d}x + \int_\Omega\left[\frac{\partial a}{\partial y}(\cdot,y)+u\right]pw\mathrm{d}x
= \sum_{t\in\mathcal{D}}\langle(y(t)-y_t)\delta_t,w\rangle + \int_\Omega gw\mathrm{d}x
\end{equation}
for all $w\in W^{1,s}_0(\Omega)$, and allows us to prescribe the optimal adjoint state as a compactly supported cutoff of the fundamental solution of $-\Delta$. More precisely, setting $\kappa_t:=\bar y(t)-y_t$ and $\varrho:=|x-t|$, for $x\neq t$, we define
\[
\bar p(x) = \kappa_t\,\Theta(\varrho)\,E(\varrho), \quad E(\varrho)= -(2\pi)^{-1}\log \varrho \text{ for } d = 2, \quad E(\varrho)= (4\pi \varrho)^{-1} \text{ for } d = 3.
\]
Here, $\Theta:[0,\infty)\to[0,1]$ is the cutoff function given by $\Theta(\varrho) =1 $ for $\varrho \leq \tfrac18$, $\Theta(\varrho) = 1-S(8\varrho-1)$ for $\varrho \in [\tfrac18,\tfrac14]$, and $\Theta(\varrho) = 0$ for $\varrho\geq\tfrac14$, where $S(s)=6s^5-15s^4+10s^3$. The optimal state is $\bar y(x) = \Pi_{i=1}^d\sin(\pi x_i)$. Notice that $\bar y(t)=1\neq0$.
Moreover, since $a(\cdot,0) = \sinh(0) = 0$, we have $f - a(\cdot,0) = f > 0$ in $\Omega$. Indeed, $-\Delta\bar y=d\pi^2\bar y>0$, $\sinh(\bar y)>0$, and $\bar u\bar y\geq\mathtt{a}\bar y>0$ in $\Omega$. Thus, Theorem \ref{thm:C01_reg_baru} guarantees the Lipschitz regularity of $\bar{u}$ for $d \in \{2,3\}$, and the corresponding error bounds of Section \ref{sec:error_estimates} apply.

\subsection{Results}

Figure \ref{fig:ex-1.1} shows the experimental convergence rates (ECRs) for the $L^2(\Omega)$-error in the control approximation for both discretization schemes in two dimensions (panel (A.1)) and three dimensions (panels (A.2)--(A.3)). In panel (A.1), the ECRs behave as $O(h)$ and $O(h^2)$ for the fully discrete and semidiscrete schemes, respectively, in agreement with our theory for $d=2$. In panel (A.2), the fully discrete scheme in three dimensions exhibits an ECR consistent with $O(h)$, in agreement with our theory for $d=3$. In panel (A.3), the semidiscrete scheme in three dimensions exhibits an ECR consistent with $O(h^2)$, exceeding the rate $O(h|\log h|)$ established in Section \ref{sec:error_estimates}; this suggests that the improved error bound for $d = 2$ may extend to $d=3$.

\begin{figure}[!ht]
\centering
\includegraphics[trim={0 0 0 0},clip,width=13.00cm,height=3.8cm,scale=0.4]{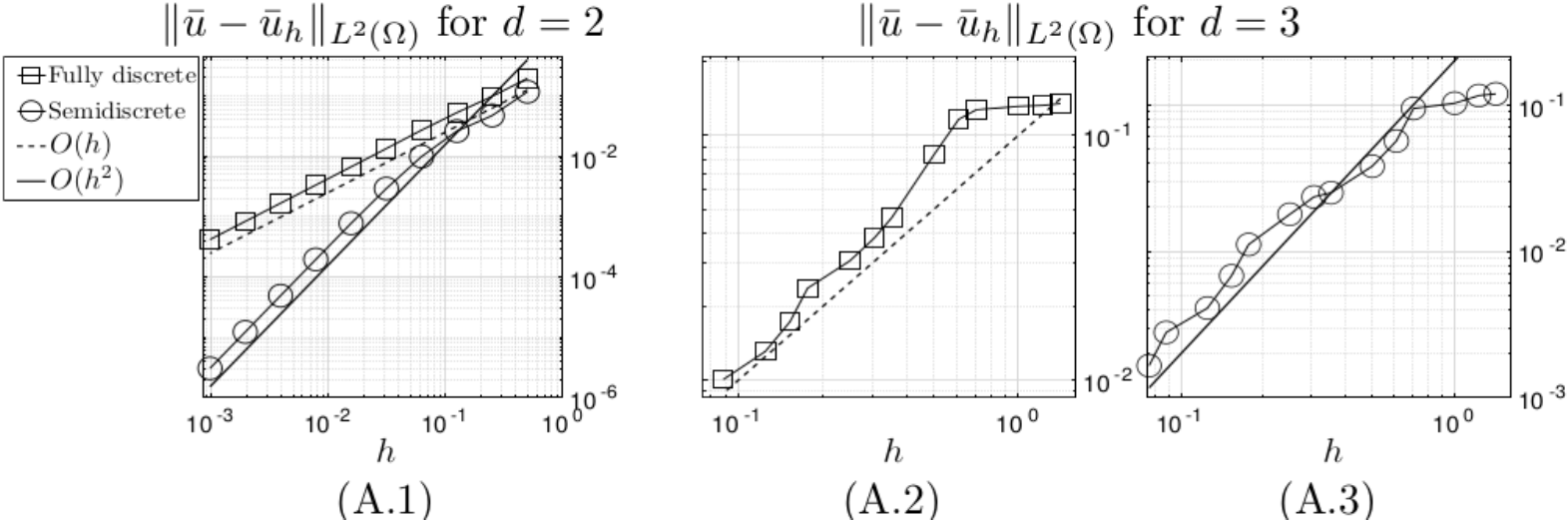}
\vspace{-0.7cm}
\caption{Experimental convergence rates for the $L^2(\Omega)$-error in the control approximation for the fully discrete and semidiscrete schemes with $d = 2$ (A.1) and $d = 3$ (A.2)--(A.3).}
\label{fig:ex-1.1}
\end{figure}

\bibliographystyle{siam}
\bibliography{bil_track_ref}

\end{document}